\documentclass[a4paper,oneside,12pt,reqno,mathscr,final]{amsart}
\usepackage{amssymb}
\usepackage{euscript}
\usepackage[frame,cmtip,curve,arrow,matrix,line,graph]{xy}

\usepackage{showkeys}

\usepackage[dvipdfm]{hyperref}

\newcommand{\bb} {\mathbb}
\newcommand{\cal}{\mathcal}
\newcommand{\frk}{\mathfrak}

\newcommand{\bbA}{\bb{A}}
\newcommand{\bbC}{\bb{C}}
\newcommand{\bbF}{\bb{F}}
\newcommand{\bbK}{\bb{K}}
\newcommand{\bbL}{\bb{L}}
\newcommand{\bbP}{\bb{P}}
\newcommand{\bbR}{\bb{R}}
\newcommand{\bbQ}{\bb{Q}}
\newcommand{\bbZ}{\bb{Z}}

\newcommand{\fb}{\frk{b}}
\newcommand{\fh}{\frk{h}}
\newcommand{\fk}{\frk{k}}
\newcommand{\fm}{\frk{m}}
\newcommand{\frg}{\frk{g}}
\newcommand{\fgl}{\frk{gl}}
\newcommand{\fsl}{\frk{sl}}

\newcommand{\bx}{\mathord{\mathbf{x}}}
\newcommand{\by}{\mathord{\mathbf{y}}}

\newcommand{\catA}{\mathord{\mathsf{A}}}

\newcommand{\catR}{\mathord{\mathsf{R}}}
\newcommand{\catS}{\mathord{\mathsf{S}}}

\newcommand{\grpS}{\mathord{\frk{S}}}

\newcommand{\shE}{\mathord{\cal{E}}}
\newcommand{\shF}{\mathord{\cal{F}}}
\newcommand{\shL}{\mathord{\cal{L}}}
\newcommand{\shO}{\mathord{\cal{O}}}
\newcommand{\shP}{\mathord{\cal{P}}}
\newcommand{\shT}{\mathord{\cal{T}}}

\newcommand{\stkM}{\mathord{\frk{M}}}
\newcommand{\stkS}{\mathord{\frk{S}}}
\newcommand{\stkT}{\mathord{\frk{T}}}
\newcommand{\stkU}{\mathord{\frk{U}}}
\newcommand{\stkV}{\mathord{\frk{V}}}
\newcommand{\stkX}{\mathord{\frk{X}}}
\newcommand{\stkY}{\mathord{\frk{Y}}}
\newcommand{\stkZ}{\mathord{\frk{Z}}}

\newcommand{\lf} {\mathrm{lf}}
\newcommand{\tw} {\mathrm{tw}}
\newcommand{\coh}{\mathrm{coh}}
\newcommand{\ext}{\mathrm{ext}}
\newcommand{\mot}{\mathrm{mot}}
\newcommand{\qpr}{\mathrm{qpr}}
\newcommand{\red}{\mathrm{red}}
\newcommand{\ses}{\mathrm{ss}}
\newcommand{\tor}{\mathrm{tor}}
\newcommand{\perf}{\mathrm{perf}}

\newcommand{\rk} {\operatorname{rk}}
\newcommand{\chr}{\operatorname{char}}
\newcommand{\sgn}{\operatorname{sgn}}

\newcommand{\FM}{\operatorname{FM}}
\newcommand{\GL}{\operatorname{GL}}
\newcommand{\SL}{\operatorname{SL}}
\newcommand{\Gr}{\operatorname{Gr}}
\newcommand{\Id}{\operatorname{Id}}

\newcommand{\Aut}{\operatorname{Aut}}

\newcommand{\End}{\operatorname{End}}
\newcommand{\Ext}{\operatorname{Ext}}
\newcommand{\Hom}{\operatorname{Hom}}
\newcommand{\Iso}{\operatorname{Iso}}

\newcommand{\Num}{\operatorname{Num}}
\newcommand{\Pic}{\operatorname{Pic}}
\newcommand{\Sym}{\operatorname{Sym}}

\newcommand{\Cone}{\operatorname{Cone}}
\newcommand{\Spec}{\operatorname{Spec}}

\newcommand{\Lb}{\mathop{L\frk{b}_+}}
\newcommand{\Lsl}{\mathop{L\frk{sl}_2}}
\newcommand{\wot}{\mathop{\widehat{\otimes}}}

\newcommand{\Coh} {\mathop{\mathsf{Coh}}\nolimits}
\newcommand{\Tor} {\mathop{\mathsf{Tor}}\nolimits}
\newcommand{\Torf}{\mathop{\mathsf{Tor^{\perf}}}\nolimits}
\newcommand{\DCoh}{\mathop{\mathsf{D^b_{\coh}}}}
\newcommand{\RHom}{\mathop{\mathbf{R}\mathrm{Hom}}}

\newcommand{\ol}{\overline}

\newcommand{\wt}{\widetilde}
\newcommand{\wh}{\widehat}

\newcommand{\ep}{\epsilon}

\newcommand{\algU}{\mathord{\ddot{\mathrm{U}}}}

\newcommand{\Ca} {\mathord{C}}
\newcommand{\Csm}{\mathord{C_{\mathrm{sm}}}}
\newcommand{\Ea} {\mathord{E}}
\newcommand{\Esm}{\mathord{E_{\mathrm{sm}}}}

\newcommand{\U}[1]{\mathord{\mathrm{U}(#1)}}
\newcommand{\Ue}[1]{\mathord{\mathrm{U}_{\ext}(#1)}}
\newcommand{\DU}[1]{\mathord{\mathrm{DU}(#1)}}

\newcommand{\tDU}[1]{\mathord{\widetilde{\mathrm{DU}}(#1)}}

\newcommand{\Ht}[1]{\mathord{\mathrm{H}_{\tw}(#1)}}
\newcommand{\He}[1]{\mathord{\mathrm{H}_{\ext}(#1)}}
\newcommand{\Hall}[1]{\mathord{\mathrm{H}(#1)}}

\newcommand{\Dr}[1]{\mathord{\mathrm{D}_{\red}#1}}

\newcommand{\St}[1]{\mathord{\mathsf{St}/#1}}
\newcommand{\Var}[1]{\mathord{\mathsf{Var}/#1}}
\newcommand{\longto}[1]{\xrightarrow{\ #1 \ }}

\newcommand{\simto}{\xrightarrow{\sim}}
\newcommand{\longsimto}{\longto{\sim}}

\theoremstyle{plain}
 \newtheorem{thm}{Theorem}[section]
 \newtheorem{lem}[thm]{Lemma}
 \newtheorem{prop}[thm]{Proposition}
 \newtheorem{cor}[thm]{Corollary}
 \newtheorem*{thm*}{Theorem}
\theoremstyle{definition}
 \newtheorem{dfn}[thm]{Definition}
 \newtheorem{fct}[thm]{Fact}
\theoremstyle{remark}
 \newtheorem{rmk}[thm]{Remark}

\numberwithin{equation}{section}
\allowdisplaybreaks

\begin{document}

%%%%% Authors & Paper data %%%%%

\title[Quantum toroidal algebras and motivic Hall algebras I]{Quantum toroidal algebras and motivic Hall algebras I.\ Hall algebras for singular elliptic curves}
\author{Shintarou Yanagida}
\address{Research Institute for Mathematical Sciences,
Kyoto University, Kyoto 606-8502, Japan}
\email{yanagida@kurims.kyoto-u.ac.jp}

%\subjclass[2010]{16T10,17B37}
\date{April 23, 2014}

%%%%% front matter %%%%%

\begin{abstract}
We consider the motivic Hall algebra of coherent sheaves 
over an irreducible reduced projective curve of arithmetic genus $1$.
We introduce the composition subalgebra in the singular curve case,
and show that it is isomorphic to the composition subalgebra for a smooth elliptic curve.
As in the case of smooth elliptic case studied by Burban and Schiffmann,
the reduced Drinfeld double of the composition subalgebra 
is isomorphic to the quantum toroidal algebra for $\mathfrak{gl}_1$
(also called Ding-Iohara-Miki algebra), 
and it inherits automorphisms induced from equivalences of 
the associated derived category.
We show that one of the non-trivial automorphisms coincide with the one 
constructed by Miki in a purely algebraic manner.
\end{abstract}

\maketitle

%%%%%%%%%%%%%%%%%%%%%%%%%%%%%%%%%%%%%%%%%%%%%%%%%%
%%%%%%%%%%%%%%%%%%%%%%%%%%%%%%%%%%%%%%%%%%%%%%%%%%
%%%%%%%%%%%%%%%%%%%%%%%%%%%%%%%%%%%%%%%%%%%%%%%%%%
\section{Introduction}

This paper is the first part of the study 
on the relationship between 
the quantum toroidal algebras and the motivic Hall algebras 
of projective curves of arithmetic genus one.
In this paper we prepare some generalities on motivic Hall algebras 
and restate known result on the Ringel-Hall algebras for curves.
The new results given in this paper is relatively few.

This paper arises from the investigation of 
what should be called the quantum toroidal algebra for $\fgl_1$.
We will denote it by $\algU$.
It has two parameters $q_1$ and $q_2$.
This name comes from the quantum toroidal algebra for $\fgl_n$ ($n \ge 2$) 
introduced in the work \cite{GKV:1995} of Ginzburg, Kapranov 
and Vasserot in the middle 1990s.

The algebra $\algU$ has several other names.
As far as we know, it was first introduced in 
the work \cite{BS:2012} of Burban and Schiffmann
with the name (reduced Drinfeld double of) elliptic Hall algebra.
Miki \cite{Mi:2007} called $(q, \gamma)$-analog of $W_{1+\infty}$,
In the paper \cite{FHHSY}, 
B.~Feigin, Hashizume, Hoshino, Shiraishi and the author
called it the Ding-Iohara algebra.
In the paper \cite{FT:2011}, B.~Feigin and Tsymbaliuk
also used the same name.
The papers \cite{FFJMM1,FFJMM2} of 
B.~Feigin, E.~Feigin, Jimbo, Miwa and Mukhin
called quantum continuous $\fgl_\infty$.

The structure of $\algU$ looks complicated when it is considered as
an analogue of quantum affine algebra.
However, following the approach of \cite{BS:2012},
we have rather clear understanding of this algebra 
when we consider it as the reduced Drinfeld double of 
the Ringel-Hall algebra for a smooth elliptic curve $\Esm$ defined over 
a finite field $\bbF_q$. 
It has a natural $\bbZ^2$-grading, 
where $\bbZ^2$ appears as $\Num(\Coh(\Esm))$, 
the numerical Grothendieck group of the category $\Coh(\Esm)$ of coherent sheaves.

Burban and Schiffmann also constructed an action of $\Aut(\DCoh(\Esm))$ on 
the algebra $\algU$.
They identified an action by a certain Fourier-Mukai transform with the 
with Miki's automorphism $\theta$.

Now we want to consider the Ringel-Hall algebra of coherent sheaves 
over a singular elliptic curve $\Ea$. 
The first main theorem is that 
the composition subalgebra is isomorphic to that for a smooth elliptic curve.

\begin{thm*}[{Theorem \ref{thm:Ea:qt}}]
Denote by $\U{\Ea}$ the composition subalgebra 
of the Ringel-Hall algebra for $\Ea$.
Denote by $\Dr{\U{\Ea}}$ the reduced Drinfeld double of the bialgebra $\U{\Ea}$ 
Then $\Dr{\U{\Ea}}$ is isomorphic to $\algU$,
where the parameter $q_1,q_2$ in $\algU$ 
appears as the inverse of the zeros of the motivic zeta function of $\Ea$.
Namely we put 
$$
 \zeta_{\mot}(\Ea;z) = \dfrac{(1-q_1 z)(1-q_2 z)}{(1-z)(1-q z)}.
$$
\end{thm*}

As in the case for a smooth elliptic curve studied by Burban and Schiffmann,
the reduced Drinfeld double of the composition subalgebra,
isomorphic to the so-called Ding-Iohara-Miki algebra, 
inherits automorphisms induced from equivalences of 
the associated derived category.
We will show that one of the non-trivial automorphisms coincides with the one 
constructed by Miki for Ding-Iohara-Miki algebra.

\begin{thm*}[{Theorem \ref{thm:Ea:theta}}]
There is a Fourier-Mukai transform 
inducing the algebra automorphism 
$\Phi^H = \theta$ on $\Dr{\U{\Ea}} = \algU$.
\end{thm*}

In the course of the study, 
we find that the motivic formalism of Hall algebras is convenient 
to circumvent technical issues on Fourier-Mukai transforms.
In this paper we deal with generalities on the motivic Hall algebra 
for projective curves, 
and introduce the motivic analogue of the notions around the Ringel-Hall algebras
(including the motivic version of the twisted and the extended Ringel-Hall algebras).
The theorems cited above will be stated in the motivic language.

Let us briefly spell out the organization of this paper.

In \S\ref{sect:toroidal} we give the definition of 
the $\fgl_1$-quantum toroidal algebra $\algU$.

\S\ref{sect:K} is the preparation of the Grothendieck group 
of varieties and stacks.
It will be used in the next \S\ref{sect:mh} for the introduction 
of the motivic Hall algebras.
We also introduce Kapranov's motivic zeta function,
for it will be identified with the structure constant (function)
of the quantum toroidal algebra.

In the section \S\ref{sect:mh} we also introduce motivic analog 
of the notions related to the ordinary Ringel-Hall algebras.

The sections \S\S\ref{sect:g=0} -- \ref{sect:g=1:sm} 
are restatement of the known results on the Ringel-Hall algebras 
for a smooth curve defined on a finite field.
We dare to include these facts in this paper for there is little literature 
treating concrete examples of motivic Hall algebras.

In the final section we will give the main results of the paper.
This part is relatively short since the proofs will be almost done 
in the previous sections.

%%%%%%%%%%%%%%%%%%%%%%%%%%%%%%%%%%%%%%%%%%%%%%%%%%%%%%%%%%%%
\subsection*{Acknowledgements}

The author is supported by the Grant-in-aid for 
Scientific Research (No.\ 25800014), JSPS.
This work is also partially supported by the 
JSPS for Advancing Strategic International Networks to 
Accelerate the Circulation of Talented Researchers
``Mathematical Science of Symmetry, Topology and Moduli, 
  Evolution of International Research Network based on OCAMI"''.

A part of this paper was written during the author's stay at UC Davis. 
The author would like to thank the institute for support and hospitality.

The author gave talks on this paper 
at RIMS, Kyoto University (June 2014), Kobe University (September 2014),
Nagoya University (October 2014),
the camp-style seminar of RIMS project 2014 (February 2015)
and MSJ meeting at Meiji University (March 2015).
He would like to thank organizers of these seminars and conferences.

Finally the author would also like to thank 
Professors Boris Feigin, Yoshiyuki Kimura, Ryosuke Kodera and 
Motohico Mulase for discussion and comments.

%%%%%%%%%%%%%%%%%%%%%%%%%%%%%%%%%%%%%%%%%%%%%%%%%%%%%%%%%%%%
\subsection*{Notation}
The letter $\fk$ denotes a field which will be a fixed one in each context.

We denote categories by sans-serif letters like $\catA$ and $\catS$.
For example, $\Coh(X)$ denotes the category of coherent sheaves on a scheme $X$.

We denote stacks by Fraktur letters like $\stkM$ and $\stkX$.
For example, $\stkM=\stkM(X)$ will denote the moduli stack of 
coherent sheaves on a smooth projective variety $X$. 

A coherent sheaf will be denoted by a calligraphy letter like $\shE$. 

For an abelian category $\catA$, 
the associated Grothendieck group will be denoted by $K(\catA)$.
The class of an object $a \in \catA$ in the Grothendieck group 
will be denoted by $\ol{a} \in K(\catA)$.

%%%%%%%%%%%%%%%%%%%%%%%%%%%%%%%%%%%%%%%%%%%%%%%%%%
%%%%%%%%%%%%%%%%%%%%%%%%%%%%%%%%%%%%%%%%%%%%%%%%%%
%%%%%%%%%%%%%%%%%%%%%%%%%%%%%%%%%%%%%%%%%%%%%%%%%%

\section{Quantum toroidal algebra for $\fgl_1$}
\label{sect:toroidal}

We follow 
This definition is given in 
the paper \cite{FJMM:2013} of B. Feigin, Jimbo, Miwa and Mukhin.
%Branching rules for quantum toroidal $gl_n$,
%arXiv:1409.2147.

%%%%%%%%%%%%%%%%%%%%%%%%%%%%%%%%%%%%%%%%%%%%%%%%%%
%%%%%%%%%%%%%%%%%%%%%%%%%%%%%%%%%%%%%%%%%%%%%%%%%%
\subsection{Definition}

Let $d$ and $q$ be complex numbers such that 
$$
 q_1 := d q^{-1},\quad 
 q_2 := q^2,\quad
 q_3 := d^{-1} q^{-1}
$$
satisfies 
$$
 q_1^{n_1} q_2^{n_2} q_3^{n_3} = 1 \text{ for } n_1,n_2,n_3 \in \bb{Z}
 \iff
n_1 = n_2 = n_3.
$$

%Let $q_1, q_2, q_3$ be complex parameters such that 
%$q_1 q_2 q_3 = 1$ and
%$q_1^{n_1} q_2^{n_2} q_3^{n_3} = 1$ for $n_1,n_2,n_3 \in \bb{Z}$
%holds only if
%$n_1 = n_2 = n_3$.

\begin{dfn}
The quantum toroidal algebra for $\fgl_1$,
denoted by $\algU$, 
is an associative $\bbC$-algebra 
generated by
$$
 E_{k},\ F_{k},\ H_{r},\ K^{\pm1},\ q^{\pm c} \quad
 (k \in \bb{Z},\ r \in \bb{Z}/\{0\}).
$$
with the following defining relation. 
\begin{align*}
&K K^{-1} = K^{-1} K = 1,\quad
 q^{\pm c} \text{ are central}, \quad
 q^{c}q^{-c} = q^{-c} q^{c} = 1 ,
\\
&K^{\pm}(z)K^{\pm}(w) 
=K^{\pm}(w)K^{\pm}(z) ,
\\
&\dfrac{g(q^{-c}z,w)}{g(q^{c}z,w)}
 K^{-}(z)K^{+}(w) 
=\dfrac{g(w, q^{-c}z)}{g(w, q^{c}z)}
 K^{+}(w)K^{-}(z) ,
\\
&g(z,w)K^{\pm}(q^{(1 \mp 1)c/2}z) E(w) 
+g(w, z)E(w)K^{\pm}(q^{(1\mp1)c/2}z)=0 ,
\\
&g(w,z)K^{\pm}(q^{(1 \pm 1)c/2}z) F(w) 
+g(z,w) F(w)K^{\pm}(q^{(1\pm1)c/2}z)=0 ,
\\
&[E(z), F(w)] 
=\dfrac{1}{q - q^{-1}} 
 (\delta(q^{c}w/z)K^{+}(z) 
 -\delta(q^{c}z/w)K^{-}(w)) ,
\\
&g(z,w)E(z)E(w) + g(w,z)E(w)E(z) = 0 ,
\\
&g(w, z)F(z)F(w) + g(z,w)F(w)F(z) = 0
\\
&\Sym_{z_1,z_2}[X(z_1),[X(z_2),X(w)]_q]_{q^{-1}}=0 
 \quad (X=E \text{ or } F).
\end{align*}
%and Serre-like relations.
Here we used the current expression 
\begin{align*}
&E(z) := \sum_{k \in \bb{Z}}E_{k} z^{-k},\quad
 F(z) := \sum_{k \in \bb{Z}}F_{k} z^{-k},\\
&K^{\pm}(z) := K^{\pm1} 
 \exp\Bigl(\pm(q-q^{-1})\sum_{r=1}^{\infty}H_{\pm r} z^{\mp r}\Bigr).
\end{align*}
We also used the functions $g(z,w)$ given by 
\begin{align*}
 g(z,w) := (z-q_1 w)(z-q_2 w)(z-q_3 w).
\end{align*}
Finally, in the last line 
we used the symbol $\Sym_{z,w}$
for the symmetrizer with respect to $z,w$,
and $[X,Y]_q:=X Y - q Y X$.
\end{dfn}

$\algU$ can also be considered as an associative algebra 
defined over $\bbQ(q_1,q_2)$ with $q_3 := q_1^{-1}q_2^{-1}$.
In \S\ref{sect:g=1:sm}, we will treat $\algU$ in this way.

The algebra $\algU$ is $\bbZ^2$-graded by the degree assignment
\begin{align*}
&\deg(E_{k}) = ( 1, k),\quad 
 \deg(F_{k}) = (-1, k),\quad 
 \deg(H_{r}) = ( 0, r),\\
&\deg(K) = \deg(K^{-1}) = \deg(q^{\pm c}) = (0, 0),
\end{align*}

\begin{figure}[htbp]
{\unitlength 0.1in%
}%
\caption{Grading on $\algU$}
\end{figure}

The algebra $\algU$ has also a formal coproduct
\begin{align*}
&\Delta E(z) = E(z) \otimes 1 + K^{-}(C_1z) \otimes  E(C_1z),
 \quad
 \Delta F(z) = F(C_2z) \otimes K^{+}(C_2z) + 1 \otimes F(z),
\\
&\Delta K^{+}(z) = K^{+}(z) \otimes K^{+}(C^{-1}_1 z),
 \quad
 \Delta K^{-}(z) = K^{-}(C^{-1}_2 z) \otimes  K^{-}(z),
 \quad
 \Delta q^c = q^c \otimes q^c
\end{align*}
with $C_1 := q^c \otimes 1$ and $C_2 := 1 \otimes q^c$.
This coproduct gives $\algU$ the structure of formal bialgebra.

%%%%%%%%%%%%%%%%%%%%%%%%%%%%%%%%%%%%%%%%%%%%%%%%%%
%%%%%%%%%%%%%%%%%%%%%%%%%%%%%%%%%%%%%%%%%%%%%%%%%%
\subsection{Automorphism}

In \cite[Theorem 2.1]{Mi:2007}, 
Miki constructed an (algebra) automorphism $\theta$ 
of $\algU$ such that
\begin{align*}
&E_{0} \mapsto - q^c H_{-1}, \quad
 F_{0} \mapsto a q^c H_{1}, \\
&H_{1} \mapsto    E_{0}, \quad
 H_{-1}\mapsto -a F_{0}, \\
&q^c \mapsto K,\quad
 K \mapsto q^{-c}
\end{align*}
for some coefficient $a$.
The automorphism $\theta$ exchanges the Heisenberg subalgebras 
$\frk{a}_h$ and $\frk{a}_v$ 
living on the axes in the $\bbZ^2$ (expressing the grading of the algebra).
%$$
% \theta \circ v = h,\quad 
% \theta \circ h = v \circ \tau \circ \sigma,
%$$
%where $\sigma$ and $\tau$ are
%anti-automorphisms of $U_q(\widehat{sl}_n)$ given by
%\begin{align*}
%&\sigma: e_i \mapsto e_i, \quad f_i \mapsto f_i, \quad t_i \mapsto t_i^{-1}\\
%&\tau: x_{i,k}^{\pm} \mapsto x_{i,-k}^{\pm},
%       h_{i,r} \mapsto -q^{r c}h_{i,-r}, 
%       k_i \mapsto k_i^{-1},
%       q^c \mapsto q^c.
%\end{align*}

%Burban and Schiffmann \cite{BS:2012} constructed $\theta$
%in 

\begin{figure}[htbp]
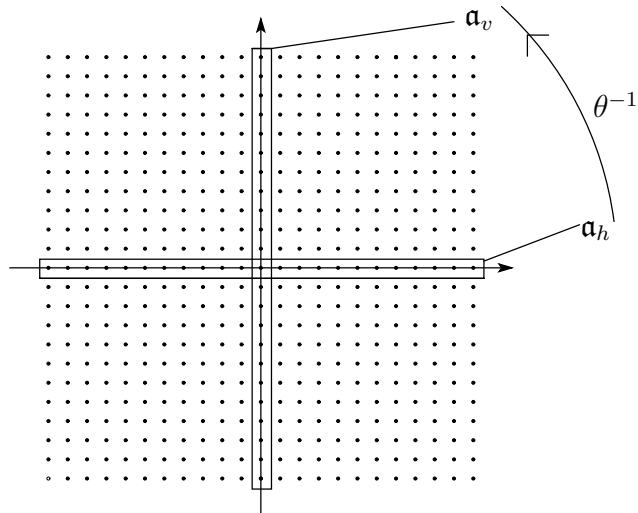

{\unitlength 0.1in%
}%
\caption{The automorphism $\theta$}
\end{figure}

%%%%%%%%%%%%%%%%%%%%%%%%%%%%%%%%%%%%%%%%%%%%%%%%%%
%%%%%%%%%%%%%%%%%%%%%%%%%%%%%%%%%%%%%%%%%%%%%%%%%%
%%%%%%%%%%%%%%%%%%%%%%%%%%%%%%%%%%%%%%%%%%%%%%%%%%
\section{The Grothendieck ring of algebraic varieties and stacks}
\label{sect:K}

This section is devoted to the introduction of the Grothendieck rings 
of varieties and stacks,
which will be used as the coefficient ring of the motivic Hall algebra.

%%%%%%%%%%%%%%%%%%%%%%%%%%%%%%%%%%%%%%%%%%%%%%%%%%
%%%%%%%%%%%%%%%%%%%%%%%%%%%%%%%%%%%%%%%%%%%%%%%%%%
\subsection{Definitions}

Let $\fk$ be an arbitrary field.
By a ``variety over $\fk$", 
we mean a reduced and separated scheme of finite type over $\fk$.
Denote by $\Var{\fk}$ the category of varieties over $\fk$.

\begin{dfn}
The \emph{Grothendieck group $K(\Var{\fk})$ of varieties over $\fk$} 
is the quotient of the free abelian group on isomorphism classes of varieties over $\fk$,
by the relations of the form
$$
 [X] = [Y] + [X \setminus Y],
$$
where $Y$ is a closed subvariety of the variety $X$ 
and $X \setminus Y$ is the complementary open subvariety.
\end{dfn}

We consider the product on $K(\Var{\fk})$ given by 
$$
 [X] \cdot [Y] = [(X \times Y)_{\red}],
$$
where the product is understood to be over $\Spec(\fk)$.
It makes $K(\Var{\fk})$ a commutative associative ring with the unit $1 = [\Spec(\fk)]$.
Hereafter we call this ring by 
the \emph{Grothendieck ring of varieties over $\fk$}.

We denote by 
$$
  \bbL := [\bbA^1] \in K(\Var{\fk})
$$
the class of the affine line.

Let us see some examples of computations.

\begin{lem}
\label{lem:bbL} 
\begin{enumerate}
\item 
For the $d$-th general linear group $\GL_d$, 
$$
 [\GL_d] = \bbL^{d(d-1)/2} \prod_{k=1}^{d}(\bbL^{k}-1) 
         = \bbL^{d(d-1)/2} (\bbL-1)^d [d]^+_{\bbL}!,
$$
where $[n]^+_q := 1+q+\cdots +q^{n-1}$ 
and $[n]^+_q! := [1]^+_q \cdot [2]^+_q \cdot \cdots \cdot [n]^+_q$
for $n \in \bbZ_{\ge 0}$ and $q$ an indeterminate.

\item
For the Grassmann variety $\Gr(d,n)$ with $0 \le d \le n$, 
$$
 [\Gr(d,n)] = {\genfrac{[}{]}{0pt}{}{n}{d}}^+_{\bbL} := 
 \dfrac{[n]^+_{\bbL}!}{[d]^+_{\bbL}! [n-d]^+_{\bbL}!}.
$$
\end{enumerate}
\end{lem}

\begin{proof}
(1) 
See \cite[Lemma 2.6]{B:2012}.

(2)
Considering the transitive action of $\GL_n$ on $\Gr(d,n)$,
we have $[\Gr(d,n)] = [\GL_n] \cdot [\GL_{n,d}]$,
where $\GL_{n,d}$ is the isotropy subgroup of the fixed $d$-dimensional 
subspace $k^d$ of the total $n$-dimensional space $k^n$.
Since $\GL_{n,d} \simeq \GL_d \times \GL_{n-d} \times M_{d,n-d}$,
where $M_{p,q}$ is the collection of $p \times q$ matrices,
we have
\begin{align*}
 [\GL_n] &= [\Gr(d,n)] \cdot [\GL_{n,d}] \\
 %[\GL_d] \cdot [\GL_{n-d}] \cdot [M(d \times (n-d))] \\
         &= [\Gr(d,n)] \cdot [\GL_d] \cdot [\GL_{n-d}] \cdot \bbL^{d(n-d)}.
\end{align*}
Using (1), one can check the result immediately.
\end{proof}

%%%%%%%%%%%%%%%%%%%%%%%%%%%%%%%%%%%%%%%%%%%%%%%%%%
%%%%%%%%%%%%%%%%%%%%%%%%%%%%%%%%%%%%%%%%%%%%%%%%%%
\subsection{Kapranov's motivic zeta function}

Let us now introduce another example of the calculation in $K(\Var{\fk})$,
namely Kapranov's motivic zeta function \cite{K}.
This function is important since it will appear 
in the study of motivic Hall algebra of curve
as the ``structure function". 

Let us recall the symmetric product of a variety.
Let $X$ be a quasi-projective variety over $\fk$. 
For every $n \in \bbZ_{\ge 1}$, 
we have a natural action of the symmetric group $\grpS_n$ on the product $X^n$. 
Since $X^n$ is again quasi-projective,
the quotient of $X^n$ by $\grpS_n$ exists. %(SGA1).
This is the $n$-th symmetric product of $X$, which we denote by $\Sym^n(X)$.
We define $\Sym^0(X) := \Spec(k)$ for convention.

If $\fk$ is perfect, then $X^n$ is reduced, so that $\Sym^n(X)$ is also reduced. 

Kapranov \cite{K} introduced the motivic analogue of the Hasse-Weil zeta functions 
for varieties as the generating series of the classes of the symmetric products 
in the Grothendieck ring $K(\Var{\fk})$.
Before writing down the definition, we follow \cite{Mu} to introduce 
a description of $K(\Var{\fk})$ 
in terms of quasi-projective varieties.

\begin{dfn}
Let $\wt{K}(\Var{\fk})$ be the quotient of $K(\Var{\fk})$ 
by the subgroup generated by the relations $[X] - [Y]$, 
where there is a a radicial surjective morphism $X \to Y$ of varieties over $\fk$. 
\end{dfn}

Note that $\wt{K}(\Var{\fk})$ is a quotient ring of $K(\Var{\fk})$.
This is because if $f: X \to Y$ is surjective and radicial, 
then for every variety $Z$, the morphism $f \times \Id_Z: X \times Z \to Y \times Z$ 
is surjective and radicial, for $f \times \Id_Z$ is the base-change of $f$ 
with respect to the projection $Y \times Z \to Y$.

If $\chr(k) = 0$, 
then the canonical quotient map $K(\Var{\fk}) \to \wt{K}(\Var{\fk})$ is an isomorphism.

\begin{dfn}
\begin{enumerate}
\item 
Let $K^{\qpr}(\Var{\fk})$ be the quotient of the free abelian group 
on isomorphism classes of quasi-projective varieties over $\fk$, modulo the relations
$$ 
 [X] = [Y] + [X \setminus Y]
$$
where $X$ is a quasi-projective variety and $Y$ is a closed subvariety of $X$. 
We have a group homomorphism from $K^{\qpr}(\Var{\fk})$  to $K(\Var{\fk})$.
Denote it by 
$$
 \Phi: K^{\qpr}(\Var{\fk}) \to K(\Var{\fk}).
$$

\item
Define $\wt{K}^{\qpr}(\Var{\fk})$ to be the quotient of $K^{\qpr}(\Var{\fk})$ by
the relations $[X] - [Y]$, where we have a surjective, radicial morphism of 
quasi-projective varieties $f: X \to Y$. 
Denote the corresponding group homomorphism by 
$$
 \wt{\Phi}: \wt{K}^{\qpr}(\Var{\fk}) \to \wt{K}^{\qpr}(\Var{\fk}).
$$
\end{enumerate}
\end{dfn}

\begin{fct}[{\cite[Proposition 7.27]{Mu}}]
If $\chr(k)=0$, then 
$\Phi$ and $\wt{\Phi}$ are both isomorphisms.
\end{fct}

Now we can explain the definition of the motivic zeta function.

\begin{dfn}[\cite{K,Mu} ]
For a quasi-projective variety $X$ over a field $\fk$, 
the \emph{motivic zeta function} of $X$ is defined to be
$$
 \zeta_{\mot}(X;z) := \sum_{n \ge 0} [\Sym^n(X)] z^n
 \in 1 + z \cdot \wt{K}(\Var{\fk})[[z]].
$$
\end{dfn}

\begin{fct}[{\cite[Proposition 7.28]{Mu}}]
Assume $\fk$ is perfect.
The map $[X] \to \zeta_{\mot}(X;z)$  for $X$ a quasi-projective variety 
defines a group homomorphism $K(\Var{\fk}) \to 1 + t \cdot \wt{K}(\Var{\fk})[[t]]$,
which factors through $\wt{K}(\Var{\fk})$.
\end{fct}

%\begin{rmk}
%If $\fk$ is not perfect, then we need to use the reduced scheme $\Sym^n(X)_\red$ 
%of $\Sym^n(X)$ and change the definition of $\zeta_{\mot}$ to be
%$\zeta_{\mot}(X;t) := \sum_{n \ge 0} [\Sym^n(X)_\red] t^n$.
%\end{rmk}

Here is the rationality result of the motivic zeta function for a smooth projective curve,
originally proved by Kapranov.

\begin{fct}[{\cite[(1.1.9) Theorem]{K}, \cite[Theorem 7.33]{Mu}}]
\label{fct:Mustata:thm}
Let $\fk$ be a perfect field. 
If $X$ is a smooth, geometrically connected, 
projective curve of genus $g$ over $\fk$ 
which has a $\fk$-rational point, 
then $\zeta_{\mot}(X;z)$ is a rational function. 
Moreover, we have
$$
 \zeta_{\mot}(X;z) = \dfrac{f(z)}{(1-z)(1-\bbL z)}
$$
for a polynomial $f$ of degree $\le 2g$ with coefficients in $\wt{K}(\Var{\fk})$.
\end{fct}

\begin{proof}
One can prove this statement similarly as in the case of the classical Hasse-Weil zeta function.
Assume for simplicity that $\fk$ is algebraically closed, 
so that a line bundle of degree $1$ 
exists on $X$.
Recall that the Abel map
$a:\Sym^d(X) \to \Pic^0(X)$ 
to the Picard variety of degree $0$
is the projective bundle with the fiber 
$a^{-1}(\shL) \cong \bbP(H^0(X,\shL)^\vee)$ 
for each $\shL \in \Pic^0(X)$, 
so by the Riemann-Roch formula 
the fiber is isomorphic to $\bbP^{d-g}_k$ for $d\ge2g-1$,
where $\bbP^n_k$ denotes the projective space of dimension $n$ over $\fk$.
Thus we have 
\begin{align*}
 \zeta_{\mot}(X;z)
& = \sum_{d = 0}^{2g-2}[\Sym^d(X)] z^d
  + \sum_{d \ge \min\{2g-1,0\}}[\Pic^0(X)] \cdot [\bbP^{d-g}_k] z^d.
\end{align*}
If $g \ge 1$, then a calculation gives 
\begin{align*}
 \zeta_{\mot}(X;z)
& = \sum_{d = 0}^{2g-2}[\Sym^d(X)] z^d
    + \dfrac{[\Pic^0(X)] }{(1-z)(1-\bbL z)}
      \dfrac{(\bbL^{1-g}-1)+ z (\bbL-\bbL^{1-g}) }{\bbL-1},
\end{align*}
and this expression gives the consequence.
The case $g=0$ is
\begin{align}\label{eq:Zmot:g=0}
 \zeta_{\mot}(\bbP^1;z) = \dfrac{1}{(1-z)(1-\bbL z)}.
\end{align}
For the proof in the case of the general base field $\fk$, see \cite{Mu}.
\end{proof}

%Let us also mention
%
%\begin{fct}[{\cite[Proposition 7.26]{Mu}}] 
%If $k = \bbF_q$ is a finite field, then the ring homomorphism
%$K(\Var{\fk}) \to \bbZ$ given by $[X] \to |X(\bbF_q)|$ 
%factors through $\wt{K}(\Var{\fk})$
%\end{fct}

%%%%%%%%%%%%%%%%%%%%%%%%%%%%%%%%%%%%%%%%%%%%%%%%%%
%%%%%%%%%%%%%%%%%%%%%%%%%%%%%%%%%%%%%%%%%%%%%%%%%%

\subsection{Grothendieck ring of stacks}

By the word `stack', we mean an Artin stack which is locally of finite type over $\fk$ 
unless otherwise stated.
We denote by $\St{\fk}$ the 2-category of Artin stacks of finite type over $\fk$.
Given a scheme $S$ over $\fk$ and a stack $\stkX$, 
we denote by $\stkX(S)$ the groupoid of $S$-valued points of $\stkX$.

\begin{dfn}
\begin{enumerate}
\item 
A stack $\stkX$ locally of finite type over $\fk$ is said to have \emph{affine stabilizers} 
if for every $\fk$-valued point $x \in X(\fk)$
the group $\Iso_{\fk}(x, x)$ of isomorphisms is affine.

\item
A morphism $f: \stkX \to \stkY$ in the category $\St{\fk}$ will be called 
a \emph{geometric bijection}
if it is representable and the induced functor on groupoids of $\fk$-valued points
$f(\fk): \stkX(\fk) \to \stkY(\fk)$ is an equivalence of categories.
\end{enumerate}
\end{dfn}

\begin{dfn}
Let $K(\St{\fk})$ be the free abelian group spanned by isomorphism classes of stacks
of finite type over $\fk$ with affine stabilizers, modulo relations
\begin{enumerate}
\item[(a)]
$[\stkX_1 \sqcup \stkX_2] = [\stkX_1] + [\stkX_2]$ for 
every pair of stacks $\stkX_1$ and $\stkX_2$,
\item[(b)]
$[\stkX] = [\stkY]$ for every geometric bijection $f: \stkX \to \stkY$,
\item[(c)]
$[\stkX_1] = [\stkX_2]$ for every pair of Zariski fibrations 
$f_i: \stkX_i  \to \stkY$ with the same fibres.
\end{enumerate}
\end{dfn}

Although our abelian group $K(\St{\fk})$ looks extremely large,
there is a comparison between $K(\St{\fk})$ and the $K(\Var{\fk})$.
Fiber product of stacks over $\fk$ gives $K(\St{\fk})$ the structure of a commutative ring. 
We have a homomorphism of commutative rings 
\begin{equation}\label{eq:Kvar-Kst}
K(\Var{\fk}) \longto{} K(\St{\fk})
\end{equation}
by considering a variety as a representable stack.
Now we cite 

\begin{fct}[{\cite[Lemma 3.9]{B:2012}}]
\label{fct:var-stack}
The morphism \eqref{eq:Kvar-Kst} induces an isomorphism of commutative rings
$$
Q: K(\Var{\bbC})\bigl[[\GL_d]^{-1} \mid d \in \bbZ_{\ge1}\bigr]  
   \longsimto K(\St{\bbC}).
$$
\end{fct}

See also \cite[Theorem 3.10]{T} and \cite[Theorem 4.10]{Joy:2007:QJM}
for the related results.

%%%%%%%%%%%%%%%%%%%%%%%%%%%%%%%%%%%%%%%%%%%%%%%%%%
%%%%%%%%%%%%%%%%%%%%%%%%%%%%%%%%%%%%%%%%%%%%%%%%%%

\subsection{Relative Grothendieck ring}

Fix an algebraic stack $\stkS$ which is locally of finite type over $\fk$
and has affine stabilizers. 
There is a $2$-category of algebraic stacks over $\stkS$. 
Let $\St{\stkS}$ denote the full subcategory consisting of objects $f: \stkX \to \stkS$ 
for which $\stkX$ is of finite type over $\fk$. 
Such an object will be said to have affine stabilizers if the stack $\stkX$ has.

\begin{dfn}
Let $K(\St{\stkS})$ be the free abelian group spanned by 
isomorphism classes of objects of $\St{\stkS}$ 
with affine stabilisers, modulo relations
\begin{enumerate}
\item[(a)]
for every pair of objects $f_1: \stkX_1 \to \stkS$ and $f_2: \stkX_2 \to \stkS$, 
a relation
$$
 [\stkX_1 \sqcup \stkX_2 \longto{f_1 \sqcup f_2} \stkS] 
=[\stkX_1 \longto{f_1} \stkS] + [\stkX_2 \longto{f_2} \stkS] 
$$

\item[(b)]
for every commutative diagram
$$
 \xymatrix{
 \stkX_1 \ar[rr]^{g}  \ar[rd]_{f_1} & & \stkX_2 \ar[ld]^{f_2} \\ & \stkS }
$$
with $g$ a geometric bijection, a relation
$$
 [\stkX_1 \longto{f_1} \stkS] = [\stkX_2 \longto{f_2} \stkS] 
$$

\item[(c)]
for every pair of Zariski fibrations $h_1: \stkX_1 \to Y$ and $h_2: \stkX_2 \to Y$ 
with the same fibers, 
and every morphism $g: \stkY \to \stkS$, a relation
$$
 [\stkX_1 \longto{g \circ h_1} \stkS] = [\stkX_1 \longto{g \circ h_2} \stkS].
$$
\end{enumerate}
\end{dfn}

The group $K(\St{\stkS})$ has the structure of a $K(\St{\fk})$-module, 
defined by setting 
$$
 [\stkX] \cdot [\stkY \longto{f} \stkS] 
= [\stkX \times \stkY \longto{f \circ p_2} \stkS]
$$
and extending linearly.

\begin{fct}[{\cite[\S3.5]{B:2012}}]
Assume that all stacks appearing have affine stabilizers.
\begin{enumerate}
\item 
A morphism of stacks $a: \stkS \to \stkT$ induces a \emph{pushforward} 
morphism of $K(\St{\fk})$-modules
$$
 a_*: K(\St{\stkS}) \longto{} K(\St{\stkT})
$$
sending $[\stkX \longto{f} \stkS]$ to $[\stkX \longto{a \circ f} \stkT]$.

\item
A morphism of stacks $a: \stkS \to \stkT$ of finite type induces a \emph{pullback} 
morphism of $K(\St{\fk})$-modules
$$
 a^*: K(\St{\stkT}) \longto{} K(\St{\stkS})
$$
sending $[\stkY \longto{g} \stkT]$ to $[\stkX \longto{f} \stkS]$,
where $f$ is the morphism appearing in the following Cartesian square.
$$
\xymatrix{
 \stkX \ar[r] \ar[d]_{f} & \stkY \ar[d]^{g} \\ \stkS \ar[r]_{a} & \stkT}
$$

\item
The pushforwards and pullbacks are functorial, i.e.,
$(b \circ a)_* = b_* \circ a_*$ and $(b \circ a)^* = a^* \circ b^*$
for composable morphisms $a,b$ of stacks with the required properties. 

\item
Given a Cartesian square 
$$
\xymatrix{
 \stkU \ar[r]^c \ar[d]_{d} & \stkV \ar[d]^{b} \\ \stkS \ar[r]_{a} & \stkT}
$$
of stacks, we have 
$$
 b^*\circ a_* = c_* \circ d^*: K(\St{\stkS}) \longto{} K(\St{\stkV}).
$$

\item
For every pair of stacks $(\stkS_1,\stkS_2)$, 
we have a morphism (K\"unneth morphism)
$$
 K: K(\St{\stkS_1}) \otimes  K(\St{\stkS_2}) \longto{} K(\St{\stkS_1 \times \stkS_2})
$$
of $K(\St{\fk})$-modules given by 
$$
 [\stkX_1 \longto{f_1} \stkS_1] \otimes [\stkX_2 \longto{f_2} \stkS_2]
 \longmapsto [\stkX_1 \times \stkX_2 \longto{f_1 \times f_2} \stkS_1 \times \stkS_2].
$$
\end{enumerate}
\end{fct}

%%%%%%%%%%%%%%%%%%%%%%%%%%%%%%%%%%%%%%%%%%%%%%%%%%
%%%%%%%%%%%%%%%%%%%%%%%%%%%%%%%%%%%%%%%%%%%%%%%%%%
%%%%%%%%%%%%%%%%%%%%%%%%%%%%%%%%%%%%%%%%%%%%%%%%%%

\section{Motivic Hall algebra}
\label{sect:mh}

In this section we recall the notion of the motivic Hall algebra 
originally given by Joyce \cite{Joy:2007:II} 
for the framework of the Donaldson-Thomas type curve-counting invariant 
\cite{Joy:2006:JLMS, Joy:2007:QJM, Joy:2006:I, JS:2012}.
We mainly follow the notations given in the paper of Bridgeland \cite{B:2012}.

Fix a field $\fk$, and let $X$ be a projective variety over $\fk$.

Denote by $\Coh(X)$ the category of coherent sheaves on $X$,
and by $\DCoh(X)$ the bounded derived category of $\shO_X$-modules 
whose cohomology sheaves are coherent and 
which are zero except for only finitely many degrees.

Recall that an object of $\DCoh(X)$ is called perfect 
if it is isomorphic in $\DCoh(X)$ to a bounded complex 
of locally free sheaves of finite rank.
A sheaf is called perfect if it is perfect as an object of $\DCoh(X)$.
If $X$ is smooth, then every object of $\DCoh(X)$ is perfect.

%%%%%%%%%%%%%%%%%%%%%%%%%%%%%%%%%%%%%%%%%%%%%%%%%%
%%%%%%%%%%%%%%%%%%%%%%%%%%%%%%%%%%%%%%%%%%%%%%%%%%
\subsection{Moduli of flags}

We begin with the setting of moduli stacks of flags of coherent sheaves on $X$.
Since we are considering projective varieties which are not necessarily smooth,
the phrase `perfect' will appear to circumvent technical issues.
As mentioned above, if $X$ is smooth, then one may remove the perfectness condition.

Let $\stkM^{(n)}=\stkM^{(n)}(X)$ denote the moduli stack of $n$-flags of 
perfect coherent sheaves on $X$. 
The objects over a scheme $S$ are chains of monomorphisms of perfect coherent sheaves 
on $S \times X$ of the form
\begin{align}\label{eq:n-flag}
   0 = \shE_0 \hookrightarrow \shE_1 \hookrightarrow \cdots 
           \hookrightarrow \shE_n = \shE
\end{align}
such that each factor $\shF_i := \shE_i/\shE_{i-1}$ is $S$-flat. 
%It follows that each sheaf $E_i$ is also $S$-flat. 
If
$$
   0 = \shE'_0 \hookrightarrow \shE'_1 \hookrightarrow \cdots 
               \hookrightarrow \shE'_n = \shE
$$
is another such object over a scheme $T$, then a morphism in $\stkM^{(n)}$ 
lying over a morphism of schemes $f: T \to S$ is a collection of isomorphisms of sheaves
$$
 \theta_i: f^*(\shE_i) \to \shE_i'
$$
such that each diagram
$$
 \xymatrix{
  f^*(\shE_i) \ar[r] \ar[d]_{\theta_i} & f^*(\shE_{i+1}) \ar[d]^{\theta_{i+1}} \\ 
  \shE_i'     \ar[r]                   &     \shE_{i+1}'}
$$
commutes.
We also denote $\stkM := \stkM^{(1)}$.

As in \cite[Lemma 4.1]{B:2012}, 
one can show that $\stkM^{(n)}$ is an algebraic stack 
using the relative Quot scheme and induction on $n$. 

There are morphisms of stacks
$$ 
 a_i: \stkM^{(n)} \longto{} \stkM
$$
for $1 \le i \le n$, 
defined by sending a flag \eqref{eq:n-flag} to its $i$-th factor 
$\shF_i = \shE_i/\shE_{i-1}$. 
%To define these it is first necessary to choose a cokernel for 
%each monomorphism $E_{i-1} \to E_i$. 
There is another morphism
$$
 b: \stkM^{(n)} \longto{} \stkM
$$
sending a flag \eqref{eq:n-flag} to the sheaf $\shE_n = \shE$.

The morphisms $a_i$'s and $b$ satisfy the iso-fibration property 
in the sense of \cite[\S A.1]{B:2012}.
A morphism $f:\stkX \to \stkY$ of stacks is called an iso-fibration 
if the following holds.
Let $S$ be any scheme and $\theta: a \to b$ be an isomorphism 
in the groupoid $\mathop{\stkY}(S)$.
Suppose there is an $a' \in \mathop{\stkX}(S)$ such that $f(a')=a$.
Then there is an isomorphism $\theta': a' \to b'$ in $\stkX(S)$ 
such that $f(\theta')=\theta$.
See \cite[Lemma A.1]{B:2012} for the claim
which yields the iso-fibration property of $a_i$'s and $b$.

Using the iso-fibration property of the morphisms $a_i$ and $b$,
we find that for $n \in \bbZ_{>1}$ there is a Cartesian square
$$
 \xymatrix{
  \stkM^{(n)}   \ar[r]^{t} \ar[d]_{s} & \stkM^{(2)} \ar[d]^{a_1} \\
  \stkM^{(n-1)} \ar[r]_{b}            & \stkM 
},
$$
where $s$ and $t$ send a flag \eqref{eq:n-flag} to the flags
$$
 \shE_1 \hookrightarrow \cdots \hookrightarrow \shE_{n-1} 
 \text{ \ and \ } 
 \shE_{n-1} \hookrightarrow \shE_n
$$
respectively.

As noted in \cite{B:2012},
the morphism $(a_1,a_2)$ is not representable.
However the following lemma holds,
and it will be used for the definition of the motivic Hall algebra.

\begin{lem}
The morphism $(a_1,a_2)$ is of finite type.
\end{lem}

\begin{proof}
The proof of \cite[Lemma 4.2]{B:2012} using regularity of sheaves 
can be applied in our situation 
since we are considering the perfect sheaves.
\end{proof}

%%%%%%%%%%%%%%%%%%%%%%%%%%%%%%%%%%%%%%%%%%%%%%%%%%
%%%%%%%%%%%%%%%%%%%%%%%%%%%%%%%%%%%%%%%%%%%%%%%%%%

\subsection{The definition of the motivic Hall algebra}

Recall that we denote by $\stkM=\stkM(X)$ the moduli stack of 
coherent sheaves over the smooth projective variety $X$.

\begin{dfn}
\begin{enumerate}
\item 
Set 
$$ \Hall{X} := K(\St{\stkM}). $$

\item
Consider the diagram 
$$
\xymatrix{ 
 \stkM^{(2)} \ar[r]^{b} \ar[d]_{(a_1,a_2)} & \stkM \\ 
 \stkM \times \stkM}
$$
where the morphisms $a_1$, $a_2$ and $b$ send 
a flag $E_1 \hookrightarrow E$
%a short exact sequence
%$$
% 0 \longto{} A_1 \longto{} B \longto{} A_2 \longto{} 0
%$$
to the sheaves $E_1$, $E/E_1$ and $E$ respectively.
Now introduce a morphism $m: \Hall{X} \otimes \Hall{X} \longto{} \Hall{X}$ 
of $K(\St{\fk})$-modules by the composition
\begin{align*}
m: \Hall{X} \times \Hall{X} 
=& K(\St{\stkM}) \times K(\St{\stkM}) \longto{K} K(\St{\stkM \times \stkM}) \\
 & \longto{(a_1,a_2)^*} K(\St{\stkM^{(2)}}) \longto{b_*} K(\St{\stkM})
   = \Hall{X},
\end{align*}
which we call the \emph{convolution product}.
We will also use the symbol $\diamond$ for 
the convolution product as a binary operator. 
\end{enumerate}
\end{dfn}

\begin{rmk}
The convolution product can be rewritten as 
$$
 [\stkX_1 \longto{f_1} \stkM] \diamond [\stkX_2 \longto{f_2} \stkM] 
 = [\stkZ \longto{b \circ h} \stkM],
$$
where $\stkZ$ and $h$ are defined by the following Cartesian square.
$$
\xymatrix{
 \stkZ \ar[rr]^{h} \ar[d]  & & \stkM^{(2)} \ar[r]^{b} \ar[d]^{(a_1,a_2)} & \stkM \\
 \stkX_1 \times \stkX_2 \ar[rr]_{f_1 \times f_2} & & \stkM \times \stkM}
$$
\end{rmk}

\begin{fct}[{\cite[Theorem 4.3]{B:2012}}]
The convolution product $m$ gives $\Hall{X}$ the structure of an associative 
unital algebra over $K(\St{\fk})$. 
The unit element is $1 := [\stkM_0 \hookrightarrow \stkM]$, 
where $\stkM_0 \simeq \Spec(k)$ is the substack of zero objects.
\end{fct}

%%%%%%%%%%%%%%%%%%%%%%%%%%%%%%%%%%%%%%%%%%%%%%%%%%
%%%%%%%%%%%%%%%%%%%%%%%%%%%%%%%%%%%%%%%%%%%%%%%%%%
\subsection{Coproduct}

As in the case of the ordinary Ringel-Hall algebra,
the motivic Hall algebra $\Hall{X}$ has a coproduct,
and if $\dim X=1$ then $\Hall{X}$ is a bialgebra.
Also one can introduce the extended algebra and the  Hall pairing
similar to the ordinary Ringel-Hall algebra.
In this subsection we introduce the motivic version of these notions.
We assume the readers are familiar to the ordinary Ringel-Hall case,
and omit some detailed arguments.
For the ordinary case, we cite \cite{G:1995} as the original literature,
and cite \cite[Lect. 1]{S:lect} as a nice review.

Hereafter let $X$ be a smooth projective variety over a field $\fk$.
The Euler form 
$$
 \chi(\shE,\shF) := \sum_i \dim_{\fk}\Ext^i(\shE,\shF)
$$
on the category $\Coh(X)$ induces a bilinear form on 
the Grothendieck group $K(\Coh(X))$.
Denote the left radical with respect to this form 
%(Serre'o'ΐ«'É'æ'èright radical 'Æ"¯ˆê)
by $K(\Coh(X))^\perp$, 
and set
\begin{align}\label{eq:NumX}
 \Num(X) := K(\Coh(X))/K(\Coh(X))^\perp,
\end{align}
which is usually called the numerical Grothendieck group.
Let
$$
 \Gamma \subset \Num(X)
$$ 
be the submonoid generated by 
the effective classes.
Throughout this paper,
an element of $K(\Coh(X))$ or $\Num(X)$ representing 
$\shE \in \Coh(X)$ is denoted by
$$
 \ol{\shE} \in K(\Coh(X)) \text{ or } \Num(X).
$$

Then the stack $\stkM=\stkM(X)$ has a decomposition
$$
 \stkM=\bigsqcup_{\alpha \in \Gamma} \stkM_\alpha.
$$
Here $\stkM_\alpha$ denotes the open and closed substack of $\stkM(X)$
whose objects have the classes $\alpha \in \Gamma$ in $\Num(X)$..

The injection $\stkM_\alpha \hookrightarrow \stkM$ induces 
a $K(\St{\fk})$-module injective homomorphism 
$K(\St{\stkM_\alpha}) \hookrightarrow K(\St{\stkM})$,
and the above decomposition yields the direct sum decomposition 
\begin{equation}\label{eq:HX:decomp}
 \Hall{X} = \bigoplus_{\alpha \in \Gamma}\Hall{X}_\alpha,\quad
 \Hall{X}_\alpha := K(\St{\stkM_\alpha})
\end{equation}
of $K(\St{\fk})$-module.
Then $\Hall{X}$ is a $\Gamma$-graded algebra with respect to $\diamond$.

Let us fix a square root $\bbL^{1/2}=\sqrt{\bbL}$ of $\bbL$,
and set
$$
 \bbK := K(\St{\fk})[\bbL^{\pm1/2}].
$$

\begin{dfn}
\begin{enumerate}
\item 
Let $\Ht{X}$ be the $\bbK$-algebra structure
on $\Hall{X}=K(\St{\stkM(X)})$ with the product
$$
 x_\alpha * x_\beta 
 := \bbL^{\chi(\alpha,\beta)/2}  x_\alpha \diamond x_\beta.
$$
Here $x_\alpha$ (resp.\ $x_\beta$) is an element of 
$\Hall{X}_\alpha$ (resp.\ $\Hall{X}_\beta$) in the decomposition  \eqref{eq:HX:decomp}.

\item 
Let $\He{X}$ be the $\bbK$-algebra
which is obtained as an extension of $\Ht{X}$ by 
$$
 \{k_\alpha \mid \alpha \in \Num(X)\}
$$
with the following relations.
\begin{align}\label{eq:Hext:*}
 k_\alpha * k_\beta = k_{\alpha+\beta},\quad
 k_\alpha * x_\beta = 
 \bbL^{\bigl(\chi(\alpha,\beta)+\chi(\beta,\alpha)\bigr)/2}
 x_\beta * k_\alpha.
\end{align}
Here $\alpha,\beta \in \Num(X)$ and 
$x_\beta \in \Hall{X}_\beta$.
\end{enumerate}
\end{dfn}

We immediately have 

\begin{lem}\label{lem:He:Gamma-graded}
The extended motivic Hall algebra $\He{X}$ is $\Gamma$-graded algebra 
with respect to $*$.
\end{lem}

Next we introduce the motivic version of the Green coproduct \cite{G:1995}.
Denote by $\wot$ the completion of the tensor product 
with respect to the grading \eqref{eq:HX:decomp}.

\begin{dfn}\label{dfn:coprod}
\begin{enumerate}
\item
Define $\Delta:\Hall{X} \longto{} \Hall{X} \wot \Hall{X}$ as 
\begin{align*}
\Delta: \Hall{X} = 
&K(\St{\stkM}) \longto{b^*}
 K(\St{\stkM^{(2)}}) 
 \longto{(a_1,a_2)_*}
 K(\St{\stkM \times \stkM}) \\
 &\longto{} 
  K(\St{\stkM}) \wot  K(\St{\stkM}) 
=\Hall{X} \wot \Hall{X}. 
\end{align*}

\item
Define $\Delta: \He{X} \longto{} \He{X} \wot \He{X}$ by 
$$
 \Delta(x k_\alpha)=\delta_\alpha\circ\Delta(x).
$$
Here $\delta_\alpha: \Hall{X}\wot \Hall{X} \longto{} \He{X} \wot \He{X}$ 
is defined as a linear extension of 
$$
 \delta_{\alpha}(x \otimes y_\beta) := x k_{\alpha+\beta} \otimes y_\beta k_\alpha
$$
with $y_\beta \in \Hall{X}_\beta$ and $x \in \Hall{X}$.
\end{enumerate}
\end{dfn}

As an immediate consequence of 
the ordinary Ringel-Hall algebra associated to a finitary abelian category 
of global dimension $1$ \cite{G:1995},
we have

\begin{prop}
If $X$ is a smooth projective curve over a field $\fk$,
then $\He{X}$ is a $\bbK$-bialgebra with respect to $*$ and $\Delta$.
\end{prop}

Next we introduce the motivic version of Green's scalar product.
%Our scalar product will be valued in the total ring of fractions $Q(\bbK)$ 
%of the ring $\bbK$.

\begin{dfn}\label{dfn:H-pairing}
Define the non-degenerate bilinear form $(\cdot,\cdot)_H$ on $\He{X}$ by
\begin{align*}
&(\cdot,\cdot)_H: \He{X} \otimes \He{X} \to \bbK, \\
&([\stkM_{\shE}] k_\alpha, [\stkM_{\shF}] k_\beta)_H := \delta_{\shE,\shF} 
 \dfrac{ \bbL^{\bigl(\chi(\alpha,\beta)+\chi(\beta,\alpha)\bigr)/2} }{a_{\shE}},
\end{align*}
where 
$\shE,\shF \in \Coh(X)$ and
$\stkM_{\shE}$ (resp.\ $\stkM_{\shF}$) denotes 
the moduli stack of coherent sheaves isomorphic to $\shE$ 
(resp.\ $\shF$).
The symbol $[\stkM_{\shE}]$ is the abbreviation of 
$[\stkM_{\shE} \hookrightarrow \stkM]$.
$[\stkM_{\shF}]$ is similar.
Finally 
$$
 a_{\shE} := [\Aut(\shE)] \in K(\Var{\fk})
$$ 
is the class in $K(\Var{\fk})$ 
of the algebraic group of automorphisms of $\shE$.
\end{dfn}

\begin{prop}
If $X$ is a smooth projective curve, then 
$(\cdot,\cdot)_H$ is a Hopf pairing, 
namely $(x*y,z)_H=(x\otimes y,\Delta z)_H$ holds.
Here the bilinear form on the tensor product is defined as 
$(x \otimes y, z\otimes w)_H:=(x,z)_H (y,w)_H$.
\end{prop}

%%%%%%%%%%%%%%%%%%%%%%%%%%%%%%%%%%%%%%%%%%%%%%%%%%
%%%%%%%%%%%%%%%%%%%%%%%%%%%%%%%%%%%%%%%%%%%%%%%%%%

\subsection{Steinitz's classical Hall algebra and the subalgebra generated by torsion sheaves}
\label{subsec:torsion}

Before starting the detailed discussion of the motivic Hall algebra $\Hall{X}$ 
for a smooth projective curve $X$,
we summarize here the result on the classical Hall algebra of Steinitz 
in terms of motivic language.
It will be necessary for the argument on the subalgebra of $\Hall{X}$ 
generated by the elements corresponding to torsion sheaves on $X$.

Consider a projective variety $X$ defined over a fixed field $\fk$.
The category $\Torf(X)$ of perfect torsion sheaves 
is abelian and closed under extension.
If $X$ is smooth, then $\Torf(X)=\Tor(X)$,
which is the category of torsion sheaves on $X$.
$\Torf(X)$ also has a decomposition
\begin{align}\label{eq:tor:decomp}
 \Torf(X) = \prod_{x}\Tor_x
\end{align}
where $x$ ranges over the set of regular points of $X$
and $\Tor_x$ is the category of torsion sheaves supported at $x$.
$\Tor_x$ is equivalent to the category of finite-dimensional modules 
over the local ring at $x$.
If moreover $\dim X=1$, then, by the regularity of $x$, 
$\Tor_x$ is also equivalent to 
the category of nilpotent representations of 
the Jordan quiver over the residue field $\fk_x$ at $x$.

Let us take a regular point $x$ of a projective variety $X$,
and denote by $\stkM_{\tor,x}$ the moduli stack of 
torsion sheaves supported on $x$.
It is a substack of $\stkM=\stkM(X)$,
and we can apply the definition and arguments of the motivic Hall algebra 
to $\stkM_{\tor,x}$.

\begin{dfn}
For a regular point $x \in X$,
let us denote by $\Hall{X}_{\tor,x}$ the motivic Hall algebra
which is $K(\St{\stkM_{\tor,x}})$ as a $\bbK$-module.
\end{dfn}

Obviously $\Hall{X}_{\tor,x}$ is a subalgebra of $\Hall{X}$. 
If $\dim X = 1$, then it is also a sub-bialgebra 
since it is closed under the coproduct $\Delta$.

Recall that the ordinary Ringel-Hall algebra of torsion sheaves supported at 
a closed point is isomorphic to Steinitz's classical Hall algebra 
(see \cite[Lect. 2]{S:lect} for instance).
Here we present the result in the motivic language.

%The torsion sheaves $\shO_{d [x]} := \shO_X/\fm_x^d$ ($d \in \bbZ_{>0}$)
%is  indecomposable, and the moduli stack $\stkM_{\shO_{d [x]}}$ 
%of torsion sheaves isomorphic to $\shO_{d [x]}$
%is a substack of $\stkM_{\tor,x}$.
%It gives rise to an element 
%\begin{align}\label{eq:HL:tdx}
% t_{d,x} := [\stkM_{\shO_{d [x]}} \hookrightarrow \stkM_{\tor,x}]
% \in 
% \Hall{X}_{\tor,x}.
%\end{align}

\begin{lem}
Let $\Csm$ be a smooth projective curve.
\begin{enumerate}
\item 
The bialgebra $\Hall{\Csm}_{\tor,x}$ is commutative and co-commutative.

\item
As an algebra 
$\Hall{\Csm}_{\tor,x}$ is isomorphic to polynomial algebra 
$\bbK[e_{1,x},e_{2,x},\ldots]$
with infinite generators $\{e_{d,x}\}_{d \in \bbZ_{\ge 1}}$. 

\item
Denote by $\fk_x$ the residue field of the point $x$.
Then for any $d \in \bbZ_{\ge 1}$ we have 
$$
 \Delta(e_{d,x}) = \sum_{r=0}^d \bbL_x^{-r (d-r)} e_{r,x} \otimes e_{n-r,x},
$$
where $\bbL_x = [\bbA^1] \in K(\Var{\fk_x})$ 
is the class of affine line defined over the field $\fk_x$.

\item
The Hall pairing $(\cdot,\cdot)_H$  on $\Hall{\Csm}_{\tor,x}$ is given by
\begin{align}\label{eq:HL:Hpair}
 (e_{m,x},e_{n,x})_H = 
 \dfrac{\delta_{m,n}}{ \bbL_x^{n} (1-\bbL_x^{-1})} 
\end{align}
\end{enumerate}
\end{lem}

\begin{proof}
For the ordinary Ringel-Hal algebra case, see for example \cite[Theorem 2.6]{S:lect},
where the representation of Jordan quiver is utilized.
The proof works in the motivic setting.
\end{proof}

Now let us turn to Steinitz's Hal algebra.
We follow \cite{M:1995} on the notations of symmetric functions.

\begin{dfn}
\begin{enumerate}
\item
Denote by $\Lambda$ the space of symmetric functions over $\bbZ$.

\item
The $n$-th elementary symmetric function 
$\sum_{i_1 < i_2 < \cdots < i_n} x_{i_1} x_{i_2} \cdots x_{i_n}$
is denoted by $e_n$.

\item
Denote by $p_n := \sum_i x_i^n$ the $n$-th power-sum symmetric function.
\end{enumerate}
\end{dfn}

Recall that $e_n$'s give rise to a basis 
$\{e_\lambda = e_{\lambda_1}e_{\lambda_2}\cdots \}_\lambda$ of $\Lambda$ 
parametrized by partitions 
$\lambda=(\lambda_1,\lambda_2,\ldots)$, 
and $p_n$'s give rise to a $\bbQ$-basis 
$\{p_\lambda = p_{\lambda_1} p_{\lambda_2} \cdots \}_\lambda$ 
of $\Lambda \otimes_\bbZ \bbQ$.
$\Lambda$ is an associative commutative algebra 
under the (usual) multiplication.
It also has a coassociative cocommutative coproduct $\Delta$ 
given by 
\begin{align}\label{eq:HL:Delta}
\Delta(e_n) = \sum_{r=0}^n e_r \otimes e_{n-r},
\end{align}
which makes $\Lambda$ into a bialgebra.
This bialgebra is also endowed with the Hopf pairing 
$(\cdot,\cdot)$ give by 
\begin{align}\label{eq:HL:HLpair}
 (p_m,p_n) = \delta_{m,n} \dfrac{n}{q^n-1}.
\end{align}
for an indeterminate $q$.

\begin{dfn}
Denote by $\Lambda_q$ the bialgebra 
$\Lambda \otimes_\bbZ \bbQ[q^{\pm1}]$,
where the product is given by the multiplication of symmetric functions
and the coproduct is given by \eqref{eq:HL:Delta}.
We always consider $\Lambda_q$ endowed with the Hall pairing \eqref{eq:HL:HLpair}.
\end{dfn}

It is also well-known that 
the Hall-Littlewood symmetric functions form 
the unique orthonormal basis of 
$\Lambda \otimes_{\bbZ} \bbQ[q^{\pm1}]$
with respect to this pairing 
subject to a triangular condition in the expansion 
with respect to monomial symmetric functions.

Comparing the bialgebra structures on $\Lambda_q$ and $\Hall{\Csm}_{\tor,x}$,
we have the following result.

\begin{cor}
If $\Csm$ is a smooth projective curve and 
$x$ is an arbitrary closed point of $\Csm$,
then we have an isomorphism of bialgebras
$$
 \phi_x: \Hall{\Csm}_{\tor,x} \longto{} \Lambda_{\bbL_x},\quad
 e_{d,x} \longmapsto \bbL_x^{-d(d-1)/2} e_d.
$$
This map is also an isometry 
in terms of the Hall pairing $(\cdot,\cdot)_H$ given in \eqref{eq:HL:Hpair} 
and the pairing $(\cdot,\cdot)$ in \eqref{eq:HL:HLpair}.
\end{cor}

Finally we will introduce some notations on the subalgebra of $\Hall{X}$ 
generated by torsion sheaves.
Let us go back to the general situation where 
$X$ is a projective variety.

\begin{dfn}
Let $X$ be a projective variety.
\begin{enumerate}
\item 
Denote by $\stkM_{\tor}$ the moduli stack of perfect torsion sheaves on $X$.

\item
Denote by $\Hall{X}_{\tor}=K(\St{\stkM_{\tor}})$ the motivic Hall algebra 
generated by perfect torsion sheaves on $X$.
\end{enumerate}
\end{dfn}

For a smooth curve $\Csm$,
$\Hall{\Csm}_{\tor}$ is a sub-bialgebra of $\Hall{\Csm}$ since 
$\Torf(\Csm) = \Tor(\Csm)$ is closed under extension and splitting.
For any projective variety $X$, 
the decomposition \eqref{eq:tor:decomp} yields 
$$
 \Hall{X}_\tor = \bigotimes_{x} \Hall{X}_{\tor,x}
$$
as $\bbK$-algebras,
where $x$ runs over the set of regular points of $X$.

We close this subsection by introduction of elements of $\Hall{X}_{\tor}$ for future use.

\begin{dfn}\label{dfn:td}
For $d \in \bbZ_{\ge 1}$ define $t_{d,x} \in \Hall{X}_{\tor,x}$ by 
\begin{align*}
 t_{d,x} := 
  \begin{cases}
    [d]_{\sqrt{\bbL}} \dfrac{\deg(x)}{d} \phi_x^{-1}(p_{d/\deg(x)}) 
     & \text{if $\deg(x) | d$} \\
   0 & \text{otherwise}
  \end{cases}
\end{align*}
and
\begin{align*}
 t_d := \sum_{x \in X} t_{d,x} \in \Hall{X}_{\tor}.
\end{align*}
\end{dfn}

\begin{dfn} 
%\begin{enumerate}
%\item 
Let $\stkM_{(0,d)}$ be the moduli stack of perfect torsion sheaves on $X$ 
with degree $d \in \bbZ_{\ge 1}$.
It is naturally a substack of $\stkM$.
Define the element $1_{(0,d)} \in \Hall{X}$ by
$$
 1_{(0,d)} := [\stkM_{(0,d)}  \hookrightarrow \stkM]. 
$$
%We also set 
%$1_{(0,0)} = t_0  %:= [\stkM_0 \hookrightarrow \stkM] 
% := 1 \in \Hall{\bbP^1}$.
\end{dfn}

The final remark is 

\begin{lem}[{\cite[Lemma 4.10]{S:lect}}]
The elements $t_d \in \Hall{X}$ satisfies
\begin{align}
\label{eq:g=0:1T}
 1 + \sum_{d \ge 1} 1_{(0,d)} z^d 
=\exp\Bigl(\sum_{d \ge 1} \dfrac{t_d}{[d]_{\sqrt{\bbL}}} z^d \Bigr).
\end{align}
\end{lem}

\begin{proof}
We copy the proof given in \cite[Lemma 4.10]{S:lect}.
In the decomposition $\Hall{X}_{\tor} \simeq \bigotimes_{x} \Hall{X}_{\tor,x}$
we have 
$$
   \exp\Bigl(\sum_{d\ge 1}\dfrac{t_d}{[d]} z^d \Bigr) 
 = \prod_{x} \exp\Bigl(\sum_{d\ge 1} \dfrac{t_{d,x}}{[d]} z^d \Bigr)
$$ 
and 
$$
 1 + \sum_{d \ge 1} 1_{(0,d)} z^d = 
 \prod_{x} \bigl(1+\sum_{d \ge 1} 1_{(0,d),x} z^d \bigr)
$$
with $1_{(0,d),x} := [\stkM_{(0,d),x}  \hookrightarrow \stkM_{\tor,x}]$,
it is enough to show that 
$$
 \exp\Bigl(\sum_{d \ge 1}\dfrac{t_{d,x}}{[d]} z^d \Bigr) 
  = 1+\sum_{d \ge 1} 1_{(0,d),x} z^d.
$$
Using  Definition \ref{dfn:td}, we have 
\begin{align*}
 \exp\Bigl(\sum_{d\ge 1}\dfrac{t_{d,x}}{[d]} z^d \Bigr) 
&= \phi_x^{-1}\Bigl(
    \exp\bigl( \sum_{\deg(x) | d} z^d p_{d/\deg(x)} \dfrac{\deg(x)}{d} \bigr) \Bigr)
\\
&= \phi_x^{-1}\Bigl(\exp\bigl( \sum_{d \ge 1} z^{d \deg(x)} \dfrac{p_{d}}{d} \bigr) \Bigr)
 = \phi_x^{-1}\bigl(1+\sum_{d \ge 1} z^{d \deg(x)} h_{d} \bigr)
\end{align*}
In the last equality we used the generating function formula of 
the complete symmetric functions $h_n$'s.
Then by the definition of $\stkM_{(0,d),x}$ we have the desired consequence.
\end{proof}

%%%%%%%%%%%%%%%%%%%%%%%%%%%%%%%%%%%%%%%%%%%%%%%%%%
%%%%%%%%%%%%%%%%%%%%%%%%%%%%%%%%%%%%%%%%%%%%%%%%%%
\subsection{The composition subalgebra and the Drinfeld double}
\label{subsec:comp-subalg}

In this subsection $C$ is a projective curve over a field $\fk$.
The motivic Hal algebra is in general too big to study,
and we want to introduce a subalgebra which is easy to handle.
We follow the work of Schiffmann and Vasserot \cite[\S6]{SV:2011} 
where they consider the subalgebra generated by certain averages of 
line bundles and torsion sheaves.

%We mimic their approach and introduce the composition algebra $\U{X}$ of $\Hall{X}$.
%Now we may define the composition subalgebra.

\begin{dfn}\label{dfn:comp-sa}
Let $\stkM_{\tor}=\stkM_{\tor}(C)$ be the moduli stack of perfect torsion sheaves on $C$
and $\stkM_{\lf,1}=\stkM_{\lf,1}(C)$ be the moduli stack of line bundles on $C$.
\begin{enumerate}
\item 
The subalgebra of 
$\He{C}$ generated by these substacks $\stkM_{\tor}$ and $\stkM_{\lf,1}$ 
are called the \emph{composition subalgebra} and denoted by $\U{C}$.

\item
The subalgebra of 
$\He{C}$ generated by these substacks and 
$\{k_\alpha\mid\alpha \in \Num(C)\}$
is called  the \emph{extended composition subalgebra} and denoted by $\Ue{C}$.
\end{enumerate}
\end{dfn}

\begin{rmk}\label{rmk:comp-sa}
\begin{enumerate}
\item
For $n \in \bbZ$, denote by $\stkM_{\lf,(1,n)}$ 
the moduli stack of line bundles of degree $n$.
Also we set 
$$
 1^{\ses}_{(1,n)} := [\stkM_{\lf,(1,n)} \hookrightarrow \stkM] \in \Hall{C}.
$$
($\lf$ denotes the word `locally free', and
 $\ses$ denotes the word `semi-stable'.)
Recall the elements $t_d \in \Hall{C}_{\tor}$ given in Definition \ref{dfn:td}.
Then $\U{C}$ is generated by 
$$
 1^{\ses}_{(1,n)} \ (n \in \bbZ), \quad 
 t_d \ (d \in \bbZ_{\ge 1}).
$$

\item
For a line bundle $\shL$ on $C$,
let us denote by $\stkM_{\shL}$ the moduli stack of coherent sheaves 
isomorphic to $\shL$.
Then we have 
$\stkM_{\lf,(1,n)} = \sqcup_{\shL \in \Pic^n(C)} \stkM_{\shL}$.
Thus $1^{\ses}_{(1,n)}$ is a summation of the elements associated to $\shL$.
It corresponds to the generator
$ %1^{\ses}_{(1,n)} = 
 \sum_{\shL \in \Pic^n(C)}[\shL]$ 
of the composition algebra in \cite{SV:2011}.
The generator of the other type 
$%1_{(0,d)} = 
 \sum_{\shT \in \Tor(C),\ \deg(\shT)=d} [\shT]$ 
appearing in \cite{SV:2011} corresponds to our $1_{(0,d)}$.
\end{enumerate}
\end{rmk}

By the definition of the coproduct $\Delta$, we have

\begin{lem}
For a smooth curve $\Csm$, 
$\U{\Csm}$ and $\Ue{\Csm}$ are closed under the coproduct $\Delta$. 
Hence they are formal $\bbK$-bialgebra.
\end{lem}

Recalling that $\Gamma \subset \Num(C)$ is the submonoid 
generated by the effective classes (see \eqref{eq:NumX} and the lines below it),
we also have

\begin{lem}
$\U{C}$ and $\Ue{C}$ are $\Gamma$-graded algebras.
\end{lem}

Next we recall the notion of Drinfeld double, 
following \cite[\S3.2]{Jos:1995} and \cite[Appendix B]{BS:2012}.

\begin{fct}[{\cite{D:1986}}]
Let $H$ be a (topological) bialgebra with a Hopf pairing $(\cdot,\cdot)_H$.
Let $H^+ := H$ be the bialgebra $H$ itself, 
and $H^-$ be the bialgebra which is isomorphic to $H$ 
as algebra and equipped with the opposite coproduct.
Finally let $D H$ be the associative algebra generated by $H^{\pm}$ 
modulo the relations
\begin{enumerate}
\item 
$H^\pm$ are subalgebras.
\item
For $a,b \in H$ 
$$
 \sum a^-_{(1)} b^+_{(2)} \bigl(a_{(2)},b_{(1)}\bigr)_H 
 = \sum b^+_{(1)} a^-_{(2)} \bigl(a_{(1)},b_{(2)}\bigr)_H.
$$
Here we used the Sweedler notation 
$\Delta (a) = \sum a_{(1)} \otimes a_{(2)}$.
\end{enumerate}
Then the multiplication map 
$H^+ \otimes H^- \to  D H$ is an isomorphism of vector spaces,
and $D H$ has a unique bialgebra structure 
such that the map $H^+ \to D H$ given by $a \mapsto a\otimes 1$ 
and the one $H^- \to D H$ given by $a \mapsto 1\otimes a$ 
are both injections of bialgebras.
\end{fct}

\begin{dfn}
We call $D H$ the \emph{Drinfeld double} of the bialgebra $H$ with respect to 
the Hopf pairing $(\cdot,\cdot)_H$.
\end{dfn}

We also need the reduced version of the Drinfeld double.
Let $\frg$ be a Kac-Moody Lie algebra and 
$\frg'$ its derived algebra.
Denote by $U_q(\fb')$ the quantum enveloping algebra 
of a Borel subalgebra $\fb' \subset \frg'$.
It is closed under the coproduct and equipped with a Hopf pairing.

\begin{fct}[{\cite{D:1986}}]\label{fct:Dr:UqLg:double}
Let $\bbC[q^{\pm1}][K_i \mid i\in I]$ be the quantised enveloping algebra 
of the Cartan subalgebra $\fh \subset \fb$.
Then we have an isomorphism of bialgebras
$$
 D U_q(\fb')/\langle K_i \otimes K_i^{-1} -1 \mid i \in I \rangle
 \simeq U_q(\frg),
$$
\end{fct}

Xiao \cite{X:1997} introduced the reduced Drinfeld double for 
the ordinary Ringel-Hall algebra.
Here we mimic his definition in the motivic case.

\begin{dfn}
Let $\Csm$ be an irreducible smooth curve and consider 
the motivic Hall algebra $\Hall{\Csm}$ and the extended algebra $\He{\Csm}$ 
equipped with the Hopf pairing $(\cdot,\cdot)_H$ 
in Definition \ref{dfn:H-pairing}.
Define the \emph{reduced Drinfeld double} $\Dr{\Hall{\Csm}}$ of $\Hall{\Csm}$  
to be %the quotient of $D H_{\ext}(X)$ by the ideal generated by 
$$
 \Dr{\Hall{\Csm}} := D \He{\Csm} / 
  \langle k_\alpha \otimes k_\alpha^{-1}-1 \mid \alpha \in \Num(X) \rangle.
$$
\end{dfn}

We may also consider the reduced Drinfeld double of the 
extended composition subalgebra $\Ue{X}$, 
since it is a bialgebra with the Hopf pairing $(\cdot,\cdot)_H$
and contains the extension part $\{k_\alpha \mid \alpha \in \Num(X)\}$.

\begin{dfn}\label{dfn:d_red-U}
Denote by $\Dr{\U{\Csm}}$ the reduced Drinfeld double of $\U{\Csm}$,
which is defined to be 
$$
 \Dr{\U{\Csm}} := D \Ue{\Csm} / 
  \langle k^+_\alpha \otimes k^-_{-\alpha}-1 \mid \alpha \in \Num(\Csm) \rangle.
$$
\end{dfn}

We immediately have 

\begin{lem}
$\Dr{\Hall{\Csm}}$ and $\Dr{\U{\Csm}}$ are $\Num(\Csm)$-graded.
\end{lem}

We also have the following triangular decomposition.

\begin{lem}
The multiplication map gives 
an isomorphism 
$$
 \U{\Csm}^{+} \otimes_{\bbK} \bbK_X \otimes_{\bbK} \U{\Csm}^{-}
 \longsimto
 \Dr{\U{\Csm}}
$$
of modules over $\bbK = K(\St{\fk})[\bbL^{\pm1/2}]$, 
where $\U{\Csm}^{\pm}$ are the copies of $\U{\Csm}$ and 
$\bbK_{\Csm} := \bbK[\Num(\Csm)]$.
\end{lem}

%%%%%%%%%%%%%%%%%%%%%%%%%%%%%%%%%%%%%%%%%%%%%%%%%%
%%%%%%%%%%%%%%%%%%%%%%%%%%%%%%%%%%%%%%%%%%%%%%%%%%
%%%%%%%%%%%%%%%%%%%%%%%%%%%%%%%%%%%%%%%%%%%%%%%%%%

\section{The case of projective line}
\label{sect:g=0}

In this section we review the work of Kapranov \cite{K:1997},
which identifies the composition subalgebra of the Ringel-Hall algebra
for the projective line $\bbP^1$ 
with the upper triangular part of the quantum affine algebra 
$U_v(\widehat{\mathfrak{sl}}_2)$.
We will present this review in terms of the motivic Hall algebra.
Several notations in this subsection are due to the review of 
Schiffmann \cite{S:lect} and the paper of Burban and Schiffmann \cite{BS:2013}.
Let us also refer the paper of Baumann and Kassel \cite{BK:2001} 
for the detailed account.

In this section we sometimes use the following symbols for $d\in\bbZ$.
$$
 [d]_{\sqrt{\bbL}} := \dfrac{\bbL^{d/2}-\bbL^{-d/2}}{\bbL^{1/2}-\bbL^{-1/2}}
  = \bbL^{-(d-1)/2}(1+\bbL+\cdots+\bbL^{d-1})
    \in \bbK = K(\St{\bbC})[\bbL^{\pm1/2}]
$$

%%%%%%%%%%%%%%%%%%%%%%%%%%%%%%%%%%%%%%%%%%%%%%%%%%
%%%%%%%%%%%%%%%%%%%%%%%%%%%%%%%%%%%%%%%%%%%%%%%%%%

\subsection{The motivic Hall algebra of projective line}

Let $\bbP^1$ be the projective line defined over 
%the complex number field $\bbC$.
a fixed field $\fk$.
In the setting of the last subsection,
set $\stkM := \stkM(\bbP^1)$, the moduli stack of coherent sheaves on $\bbP^1$.
Our algebra $\He{\bbP^1}=K(\St{\stkM})$ is 
defined on the ring $\bbK := K(\St{\fk})[\bbL^{\pm1/2}]$.
%Lemma \ref{lem:bbL} and Fact \ref{fct:var-stack} imply 
%$$
% \widetilde{K} \simeq K(\Var{\bbC})(\sqrt{\bbL}).
%$$

\begin{fct}
The indecomposable objects of the category $\Coh(\bbP^1)$ are
\begin{enumerate}
\item 
line bundles $\shO_{\bbP^1}(n)$ ($n \in \bbZ$),

\item 
torsion sheaves $\shO_{l[x]} := \shO_{\bbP^1}/\fm_x^l$, 
where $x$ is a closed point of $\bbP^1$ and $l \in \bbZ_{\ge 1}$.
\end{enumerate}
\end{fct}

%\begin{NB}
%\begin{dfn} 
%\begin{enumerate}
%\item 
%Let $\stkM_{(0,d)}$ be the moduli stack of torsion sheaves on $\bbP^1$ 
%with degree $d \in \bbZ_{\ge 1}$.
%It is naturally a substack of $\stkM$.
%Define the element $1_{(0,d)} \in \Hall{\bbP^1}$ by
%$$
% 1_{(0,d)} := [\stkM_{(0,d)}  \hookrightarrow \stkM]. 
%$$
%
%\item
%Define the elements $\{t_d \mid d \in \bbZ_{\ge 1}\} \subset \Hall{\bbP^1}$ 
%by the generating series
%\begin{align}
%\label{eq:g=0:1T}
% 1 + \sum_{d=1}^{\infty} 1_{(0,d)} z^d 
%=\exp\Bigl(\sum_{d=1}^{\infty}\dfrac{t_d}{[d]_{\sqrt{\bbL}}} z^d \Bigr).
%\end{align}
%
%\item
%Hereafter we set 
%$1_{(0,0)} = t_0  %:= [\stkM_0 \hookrightarrow \stkM] 
% := 1 \in \Hall{\bbP^1}$.
%\end{enumerate}
%\end{dfn}
%\end{NB}

Let us recall that the numerical Grothendieck group of $\Coh(\bbP^1)$ 
has the description $\Num(\Coh(\bbP^1)) \cong \bbZ^2$,
induced by the map 
$$ 
 \Coh(\bbP^1) \ni E \longmapsto (\rk(E), \deg(E)) \in \bbZ^2
$$ 
attaching the rank and degree to each coherent sheaf $E$. 
The Euler pairing is calculated as 
\begin{align}\label{eq:g=0:euler}
 \chi\bigl((r_1,d_1),(r_2,d_2)\bigr) = r_1 r_2 + r_1 d_2-r_2 d_1.
\end{align}
In particular the extended part $\{k_\alpha \mid \alpha \in \Num(X)\}$
of $\He{\bbP^1}$ is generated by the two elements  
$$ 
 k :=  k_{(1,0)}, \quad c :=k_{(0,1)}. 
$$
Thus, by Remark \ref{rmk:comp-sa},
the composition subalgebra $\U{\bbP^1}$ is generated by
$$
 1^{\ses}_{(1,n)} \ (n\in\bbZ),\quad 
 t_r \ (r\in\bbZ_{\ge 1}),\quad 
 k,\quad  c,
$$ 
where we set
$$
 1^{\ses}_{(1,n)}:= [\stkM_{\lf,(1,n)} \hookrightarrow \stkM] \in \Hall{\bbP^1}
$$
with $\stkM_{\lf,(1,n)}$ the moduli stack of line bundles of degree $n$.

Now we can explain the relation of these generators 
derived by Kapranov \cite[\S5]{K:1997},
where a general framework of the Hopf algebra of automorphic forms was used.
Let us note that 
Baumann and Kassel \cite{BK:2001} also computed the relation by giving 
explicit description of extensions of line bundles and torsion sheaves.

\begin{fct}\label{fct:g=0:rel}
The elements $1^{\ses}_{(1,n)}$, $t_r$, $k$ and $c$ satisfy the following relations.
\begin{enumerate}
\item
$c$ is central.

\item
$[k,t_r] = 0 = [t_r,t_s]$ for all $r,s \in \bbZ_{\ge 1}$.

\item
$k * 1^{\ses}_{(1,n)} = \bbL 1^{\ses}_{(1,n)} * k$ for all $n \in \bbZ$.

\item
$[t_r,1^{\ses}_{(1,n)}] = \tfrac{[2r]_{\sqrt{\bbL}}}{r} 1^{\ses}_{(1,n+r)}$ 
for all $n \in \bbZ$ and $r \in \bbZ_{\ge 1}$.

\item
$1^{\ses}_{(1,m)} * 1^{\ses}_{(1,n+1)}
  + 1^{\ses}_{(1,n)}  * 1^{\ses}_{(1,m+1)}
  = \bbL^{-1} (1^{\ses}_{(1,n+1)} * 1^{\ses}_{(1,m)} + 1^{\ses}_{(1,m+1)} * 1^{\ses}_{(1,n)})$ 
for all $m, n \in \bbZ$.
\end{enumerate}
Here $[\cdot,\cdot]$ denotes the commutator with respect to the product $*$ 
on $\He{\bbP^1}$.
\end{fct}

\begin{proof}
We sketch the outline of the proof following \cite[\S\S2--3]{BK:2001} 
and \cite[\S4.3]{S:lect}.

The relation (1), $[k,t_r]=0$ in (2), and (3) can be obtained 
by the definition 
\eqref{eq:Hext:*} of the product $*$ and the formula \eqref{eq:g=0:euler}.
%$\chi\bigl((r_1,d_1),(r_2,d_2)\bigr) = r_1 r_2 + r_1 d_2-r_2 d_1$.
Namely, from 
$\chi\bigl((0,d_1),(r_2,d_2)\bigr)+\chi\bigl((r_2,d_2),(0,d_1)\bigr)=0$
we obtain (1) and $[k,t_r]=0$.
Also from $\chi\bigl((1,0),(1,n)\bigr)+\chi\bigl((1,n),(1,0)\bigr)=2$ 
we obtain (3). 

The relation $[t_r,t_s]=0$ in (2) holds obviously since there is no non-trivial extension
between torsion sheaves.

For the relation (4), recall that an extension
of the torsion sheaf $\shO_{l[x]}$ by the line bundle $\shO_{\bbP^1}(n)$ 
is isomorphic to $\shO_{\bbP^1}(n+m) \oplus \shO_{(l-m)[x]}$: 
$$
 \xymatrix{0 \ar[r]& \shO_{\bbP^1}(n) \ar[r] 
           & \shO_{\bbP^1}(n+m) \oplus \shO_{(l-m)[x]} \ar[r] 
           & \shO_{l[x]} \ar[r] & 0}.
$$
Note that a morphism 
$f:\shO_{\bbP^1}(n) \to \shO_{\bbP^1}(n+m) \oplus \shO_{(l-m)[x]}$ 
is injective if and only if 
the image of $f$ is not included in $\shO_{(l-m)[x]}$.
Then by the formulas 
\begin{align*}
&\dim_{\fk} \Hom\bigl(\shO_{\bbP^1}(m),\shO_{\bbP^1}(n)\bigr) 
 =\max(0,n-m+1),\\
&\dim_{\fk}\Hom\bigl(\shO_{\bbP^1}(m),\shO_{l[x]}\bigr) = l
\end{align*}
of the dimension of vector spaces of morphisms,
we find that the stack of extensions of the above form is given by 
$\bbP(\bbA^{m+1}\times\bbA^{l-m}) \setminus \bbP(\bbA^{l-m})$,
whose class in $K(\Var{\fk})$ is $\bbL^{l-m}[m+1]_{\bbL}$
(by Lemma \ref{lem:bbL} (2)).
Thus we have
\begin{align}
\nonumber
1_{(0,l)} * 1^{\ses}_{(1,n)}
&=\sqrt{\bbL}^{-l} \sum_{0 \le m \le l} 
   \bbL^{l-m}[m+1]_{\bbL} [\stkM_{\shO_{\bbP^1}(n+m) \oplus \shT} \hookrightarrow \stkM]
\\
\nonumber
&=\sum_{0 \le m \le l} 
   \sqrt{\bbL}^{l-m} [m+1]_{\sqrt{\bbL}} \,
   [\stkM_{\shO_{\bbP^1}(n+m) \oplus \shT}\hookrightarrow \stkM]
\\
\label{eq:g=0:eh}
&
=\sum_{0 \le m \le l}   [m+1]_{\sqrt{\bbL}} \,
    1^{\ses}_{(1,n+m)} * 1_{(0,l-m)}
\end{align}
Here $\stkM_{\shO_{\bbP^1}(n+m) \oplus \shT}$ is the moduli stack of 
coherent sheaves isomorphic to 
$\shO_{\bbP^1}(n+m) \oplus \shT$ 
with $\shT$ a torsion sheaf of degree $l-m$.
In the last equality, we used the fact that 
the extension in the opposite way is always trivial, 
i.e., $\Ext^1(\shO_{\bbP^1}(n),\shO_{l[x]})=0$.
Then the equation \eqref{eq:g=0:eh} can be rewritten as 
$$
 \zeta_{\mot}(\bbP^1;w/\sqrt{L}z) e(z)h(w) = h(w)e(z),
$$
where we used the the generating series
$$
 e(z):=\sum_{n \in \bbZ} 1^{\ses}_{(1,n)} z^n,\quad
 h(z):=\sum_{l \ge 1} 1_{(0,l)} z^l
$$
and the motivic zeta function \eqref{eq:Zmot:g=0} for $\bbP^1$.
Now the relation (4) follows from the definition 
\eqref{eq:g=0:1T} of $t_r$.

Finally let us check the relation (5).
An extension 
\begin{align}\label{eq:g=0:ext:form1}
 \xymatrix{0 \ar[r]& \shO_{\bbP^1}(n) \ar[r]^(0.55){f} 
           & \shF \ar[r]^(0.35){g} 
           & \shO_{\bbP^1}(m)\ar[r] & 0}
\end{align}
is a non-trivial one only if $m>n$ and 
$\shF \simeq \shO_{\bbP^1}(n+p) \oplus \shO_{\bbP^1}(m-p)$,
with $p \in \bbZ$ and $1 \le p \le (m-n)/2$.

Let us consider the non-trivial case.
The space of extensions of this form with fixed $p$ 
is given by the space of non-trivial morphisms
$f: \shO_{\bbP^1}(n) \to \shO_{\bbP^1}(n+p) \oplus \shO_{\bbP^1}(m-p)$
such that the cokernel is a line bundle.
Writing $f=f_1\oplus f_2$ with 
$f_1:\shO_{\bbP^1}(n) \to \shO_{\bbP^1}(n+p)$ and
$f_2:\shO_{\bbP^1}(n) \to \shO_{\bbP^1}(m-p)$,
the condition for $f$ is 
that $f_1$ and $f_2$, considered as homogeneous polynomials in $\fk[z,w]$ 
of degrees $p$ and $m-n-p$, 
are coprime.
Let $P_{p,m-n-p}$ be the stack of such coprime polynomials $(f_1,f_2)$. 
These $P_{a,b}$ ($a,b\in\bbZ_{\ge0}$) satisfy the following relation.
$$
 \fk[z,w]_a \times \fk[z,w]_{b} =
 \coprod_{d=0}^{\min(a-d,b-d)} \bbP(\fk[z,w]_d) \times P_{a-d,b-d}.
$$
Here $\fk[z,w]_d$ denotes the set of homogeneous polynomials of degree $d$.
In \cite[Lemma 9]{BK:2001}, the number of $\bbF_q$-rational points of the set $P_{p,m-n-p}$ 
is counted by the same relation, and the answer is 
$(q-1)(q^{a+b+1}-1)$ if $a=0$ or $b=0$, 
and $(q-1)(q^2-1)q^{a+b-1}$ otherwise.
Similarly we have
\begin{align*}
 [P_{a,b}] =
 \begin{cases} 
  (\bbL-1)(\bbL^{a+b+1}-1) & \text{if $a=0$ or $b=0$}\\
  (\bbL-1)(\bbL^2-1)\bbL^{a+b-1} & \text{otherwise}
 \end{cases}.
\end{align*}
Let us come back to the counting extensions of the form \eqref{eq:g=0:ext:form1}.
Since the automorphism $\shO(n)$ is given by $\bbA^{*}$ (non-zero scalars) 
and $[\bbA^{*}]=\bbL-1$,
we see that the non-trivial extensions give the term 
$$
 (\bbL^2-1)\bbL^{m-n-1} [\stkM_{\shO(n+p)\oplus\shO(m-p)}]
$$ 
in  $1^{ss}_{(1,m)} \diamond 1^{ss}_{(1,n)}$. 
Here $\stkM_{\shO(a)\oplus \shO(b)}$ with some $a,b\in\bbZ$ denotes 
the moduli stack of the vector bundles isomorphic to 
$\shO(a)\oplus \shO(b)$.

In the case of the trivial extensions, i.e., $\shF \simeq \shO(n)\oplus \shO(m)$,
let us write the morphisms $f,g$ as 
$$
 f=f_1\oplus f_2, \quad g=g_1\oplus g_2
$$ 
with $f_1 \in \End(\shO(n))$, $f_2,g_1\in\Hom(\shO(n),\shO(m))$ 
and $g_2\in\End(\shO(m))$.
Then we see that $f_1$ and $g_2$ are automorphisms, $f_2$ is arbitrary 
and $g_1 = -g_2\circ f_2 \circ f_1$.
Thus the pair $(f,g)$ is parametrized by 
$$
 \Aut(\shO(n))\times\Aut(\shO(m))\times\Hom(\shO(n),\shO(m)).
$$
Since the equivalent relations of the trivial extensions are given 
by $\Aut(\shO(n))\times\Aut(\shO(m))$,
we see that the trivial extensions give the term 
$$
 [\Hom(\shO(n),\shO(m))]  [\stkM_{\shO(n)\oplus\shO(m)}]
 = \bbL^{m-n+1} [\stkM_{\shO(n)\oplus\shO(m)}] 
$$ 
in  $1^{ss}_{(1,m)} \diamond 1^{ss}_{(1,n)}$. 

Considering the twist factor 
$\sqrt{\bbL}^{\chi((1.m),(1,n))}=\sqrt{\bbL}^{1+m-n}$ 
in  $1^{ss}_{(1,m)} * 1^{ss}_{(1,n)}$, we finally have 
\begin{align*}
 1^{ss}_{(1,m)} * 1^{ss}_{(1,n)} &= 
  \sqrt{\bbL}^{m-n+3} [\stkM_{\shO(n) \oplus \shO(m)}] + 
  \sqrt{\bbL}^{m-n-1}(\bbL^2-1)
  \sum_{p=1}^{\lfloor (m-n)/2 \rfloor} 
    [\stkM_{\shO(n+p)\oplus \shO(m-p)}].
\end{align*}
Now the relation (5) is an equivalent expression 
of this formula.
\end{proof}

%%%%%%%%%%%%%%%%%%%%%%%%%%%%%%%%%%%%%%%%%%%%%%%%%%
%%%%%%%%%%%%%%%%%%%%%%%%%%%%%%%%%%%%%%%%%%%%%%%%%%

\subsection{Relation to the quantum affine algebra}

Let us introduce symbols for the quantum affine algebra of $\fsl_2$.
Let 
$$
 Q = \bbZ \alpha \oplus \bbZ \delta
$$ 
denote the root lattice of 
the affine Lie algebra associated to $\fsl_2$.
We begin with the observation due to Kapranov \cite{K:1997}.

\begin{fct}
The map 
$$
 K(\Coh(\bbP^1)) \longto{} Q,\quad
 \overline{\shE} \longmapsto \rk(\shE) \alpha + \deg(\shE) \delta
$$
gives an isometry of lattieces,
where the pairing on $K(\Coh(\bbP^1))$ is given by the symmetrized Euler form 
$$
 (\overline{\shE}_1,\overline{\shE}_2) :=  
 \mathop{\chi}(\shE_1,\shE_2) + \mathop{\chi}(\shE_2,\shE_1), 
$$
and the pairing on $Q$ is taken to be the Killing form.
\end{fct}

By this identification, the isomorphic classes of indecomposable objects 
in $\Coh(\bbP^1)$ gives the set of positive roots
$$
 \Phi_+ := 
  \{\alpha + n \delta \mid n \in \bbZ \} 
  \sqcup \{ n \delta \mid n\in\bbZ_{\ge 1}\}.
$$
Hereafter, it  will be called the \emph{non-standard set of positive roots} 
(following \cite[Corollary 4,4]{S:lect}).
 
Let $\Lb$ be the Borel subalgebra of the loop algebra $\Lsl$.
We understand that it our subalgebra is associated to the non-standard set 
$\Phi_+$ of positive roots. 
Now we introduce the Borel subalgebra $U_v(\Lb)$
of the quantum loop algebra $U_q(\Lsl)$ of $\fsl_2$. 
We dare to write down the definition for clarifying our argument.

\begin{dfn}\label{dfn:UvLb}
Let $v$ be an indeterminate, and $R$ be the subring of the field $\bbQ(v)$ 
consisting of the rational functions having poles only at $0$ or at roots of $1$.
The associative algebra $U_v(\Lb)$ over $R$ is generated by 
$$
 E_n,\  (n \in \bbZ), \quad
 H_l,\  (l \in \bbZ\setminus\{0\}),\quad
 K^{\pm1},\quad C^{\pm 1/2}
$$ 
with the relations
\begin{align*}
&\text{$C^{\pm1/2}$ is central,\quad $C^{1/2}C^{-1/2}=1$},\\
&[K,H_l]=[H_l,H_m]=0,\\
&K E_n = v^2 E_n K,\\
&[H_l,E_n]=\dfrac{[2l]_v}{l}  C^{-|l|/2} E_{l+n},
\\ 
&E_{m+1}E_n - v^2 E_n E_{m+1} = v^2 E_m E_{n+1} - E_{n+1} E_m
\end{align*}
with $[2l]_v := \tfrac{v^{2l}-v^{-2l}}{v-v^{-1}}$.
\end{dfn}

%This is an algebra over $\bbQ[q^{\pm1}]$.
We denote by $U_{\sqrt{\bbL}}(\Lb)$ the specialization of 
$U_{v}(\Lb)$ at $v=1/\sqrt{\bbL}$.
Precisely speaking,
we set 
$\wt{\bbQ} := \bbQ[v^{\pm1}]/(v^{-2}-\bbL) \simeq \bbQ[\sqrt{\bbL}]$
and $U_{\sqrt{\bbL}}(\Lb) := U_v(\Lb) \otimes_R \wt{\bbQ}$.

Now we can state the fundamental result.
Recall that the element $c$ in $\U{\bbP^1}$ is central. 
Let us introduce $c^{1/2}:=k_{(0,1/2)}$, which is the square root of $c$,
and set 
$$ 
 \bbK_{\bbP^1} := \bbK[c^{\pm1/2}] = K(\St{\fk})[\bbL^{\pm1/2},c^{\pm1/2}].
$$ 

\begin{fct}[{\cite{BK:2001,K:1997}}]
The assignment
\begin{align*}
&E_n \longmapsto 1^{\ses}_{(1,n)} \ (n\in\bbZ), \quad
 H_l \longmapsto t_l c^{-l/2} \ (l\in\bbZ\setminus\{0\}),\quad
 K   \longmapsto k,\quad
 C^{\pm 1/2} \longmapsto c^{\pm 1/2}
\end{align*}
extends to an isomorphism of algebras
$$
 U_{\sqrt{\bbL}}(\Lb) \otimes_{\wt{\bbQ}} \bbK
 \longto{\ \sim \ } 
 \U{\bbP^1} \otimes_{\bbK} \bbK_{\bbP^1} \subset 
 \He{\bbP^1} \otimes_{\bbK} \bbK_{\bbP^1}.
$$ 
\end{fct}

\begin{proof}
This is the consequence of Fact \ref{fct:g=0:rel} 
and the definition of $U_q(\Lb)$.
\end{proof}

%%%%%%%%%%%%%%%%%%%%%%%%%%%%%%%%%%%%%%%%%%%%%%%%%%%%%%%%%%%%
%%%%%%%%%%%%%%%%%%%%%%%%%%%%%%%%%%%%%%%%%%%%%%%%%%%%%%%%%%%%
\subsection{The Drinfeld double}

Next we want to consider the Drinfeld double.
In order to do that, we study the coproduct 
on $\He{\bbP^1}$.

\begin{fct}[{\cite[Theorem 3.3]{K:1997}}]\label{fct:g=0:Delta}
\begin{enumerate}
\item 
For any $d \in \bbZ_{\ge1}$ we have
\begin{align}\label{eq:g=0:Delta_t_l}
 \Delta(t_d) = t_d \otimes 1 + c^d \otimes t_d,\quad
\end{align}

\item
For any $n \in \bbZ$  the following equalities hold.
\begin{align}\label{eq:g=0:Delta_e_n}
 \Delta(1^{\ses}_{(1,n)}) = 1^{\ses}_{(1,n)} \otimes 1 
  + \sum_{l \ge 0} \theta_l k c^{n-l} \otimes 1^{\ses}_{(1,n-l)}.
\end{align}
Here $\theta_l \in \Hall{\bbP^1}$ ($l \in \bbZ_{\ge 0}$) 
is defined by the following generating function.
$$
 \sum_{l \ge 0} \theta_{l} z^l 
=\exp\Bigl((\bbL^{1/2}-\bbL^{-1/2})\sum_{d \ge 1} t_d z^d \Bigr).
$$
\end{enumerate}
\end{fct}

\begin{proof}
(1)
Since the category of torsion sheaves is closed under 
taking subobjects and quotients,
we immediately find from Definition \ref{dfn:coprod} of $\Delta$ that 
$$
 \Delta(1_{(0,d)})=\sum_{m\in\bbZ,\ 0\le m \le d}
  1_{(0,d-m)}k_{(0,m)} \otimes 1_{(0,m)}.
$$
Then using the relation \eqref{eq:g=0:1T} between $1_{(0,d)}$'s and $t_d$'s 
and the fact that $c$ is central,
we have the consequence.

(2)
Since censoring by line bundles preserves extensions,
it is enough to study the case $n=0$, 
namely the decomposition of the trivial line bundle $\shO$.
We sketch the outline following  \cite[Example 4.12]{S:lect}.

It is enough to count the surjection
$\shO \to \oplus_i \shO_{n_i[x_i]}$.
Each morphism $\shO \to \shO_{n_i[x_i]}$  is parametrized by  
$\Hom(\shO,\shO_{n_i[x_i]}) \setminus \Hom(\shO,\shO_{(n_i-1)[x_i]})$,
whose class in $K(\St{\fk})$ is equal to $\bbL^{2n_i}(1-\bbL^{-1})$,
so that we have 
$$
 \Delta(1^{\ses}_{(1,0)}) = 
  1^{\ses}_{(1,0)} \otimes 1 
  + \sum_{d\ge0}\bbL^{d/2}u_d k_{(1,-d)} 1^{\ses}_{(1,-d)}
$$
with 
$$
 u_d := \sum_{(x_i,n_i)}\prod_i (1-\bbL^{-1}) t_{n_i[x_i]},
$$
where $t_{n[x]} :=  [\stkM_{0, n[x]}  \hookrightarrow \stkM]$
is the element corresponding to 
the moduli stack $\stkM_{0, n[x]}$ of torsion sheaves of degree $n$ with support on $x$,
and the summation is taken over the set of tuples of distinct points $x_i$'s 
and degrees $n_i$'s with $\sum_i n_i = d$.
Finally a quick observation yields $\bbL^{d/2} u_d = \theta_d$.
\end{proof}

\begin{rmk}\label{rmk:g=0:Delta}
In the proof we don't use the property of $\bbP^1$, 
but only use the property of a smooth curve.
\end{rmk}

Next we study the Hall pairing $(\cdot,\cdot)_H$ on $\U{\bbP^1}$.

\begin{fct}[{\cite[Lemma 4.4]{BS:2013}}]\label{fct:g=0:HopfPairing}
\begin{enumerate}
\item 
$$
 (t_d,t_l)_H = \delta_{d,l}\dfrac{1}{d}
  \dfrac{[2d]_{\sqrt{\bbL}}}{\bbL^{1/2} - \bbL^{-1/2}}.
$$
\item
\begin{align}\label{eq:g=0:hp:t-th}
 (t_d,\theta_l)_H = \delta_{d,l} \dfrac{[2d]_{\sqrt{\bbL}}}{d}
\end{align}
\end{enumerate}
\end{fct}

\begin{proof}
(1)
Decomposing into the support, we have 
$t_d = \sum_{x \in \bbP^1} t_{x,d}$,
where $t_{d,x}$ is defined in Definition \ref{dfn:td}.
%(note that the right hand side is a finite sum and well-defined).
Recalling the value \eqref{eq:HL:Hpair} of the Hall pairing,
we can calculate $(t_d,t_d)_H$ as follows. 
\begin{align*}
(t_d,t_d)_H = \sum_{x \in \bbP^1} \dfrac{\deg(x)}{d}\dfrac{([d]_{\sqrt{\bbL}})^2}{\bbL^d-1}
= \dfrac{1}{d}\dfrac{([d]_{\sqrt{\bbL}})^2}{\bbL^d-1} [\bbP^1_d]
= \dfrac{1}{d} \dfrac{[2d]_{\sqrt{\bbL}}}{\bbL^{1/2} - \bbL^{-1/2}}.
\end{align*}
Here $[\bbP^1_d] \in K(\Var{\fk})$ is given by
$$
 \zeta_{\mot}(\bbP^1;z) =
   \exp\Bigl(\sum_{d \ge 1} [\bbP^1_d] \dfrac{z^d}{d} \Bigr)
$$

(2) is a restatement of (1) in terms of $\theta_l$.
See \cite[Lemma 4.4]{BS:2013} for the detail.
\end{proof}

\begin{rmk}\label{rmk:g=0:HopfPairing}
The proof will also work for an arbitrary smooth curve.
\end{rmk}

Consider the algebra 
$$
 \tDU{\bbP^1} := \Dr{\U{\bbP^1}} \otimes_{\bbK} \bbK_{\bbP}
$$
which has a triangular decomposition
$$
 \DU{\bbP^1} \otimes_{\bbK} \bbK_{\bbP} 
=\U{\bbP^1}^+ \otimes_\bbK \bbK[\Num(\bbP^1)][c^{\pm 1/2}] \otimes_\bbK
 \U{\bbP^1}^-.
$$ 
Following \cite[\S4]{BS:2013},
we replace the generator $t_d^{\pm}$ ($d\in\bbZ_{\ge 1}$) of $\U{\bbP^1}^{\pm}$ by
$$
 \wt{t}_d := t_d^{\pm} c^{\mp d/2}.
$$
Then \eqref{eq:g=0:Delta_t_l} is rewritten as
$$
 \Delta(\wt{t}^+_d) = \wt{t}^+_d \otimes c^{-d/2} + c^{d/2}\otimes \wt{t}^+_d.
$$
It is also convenient to rewrite \eqref{eq:g=0:Delta_e_n} as
$$
 \Delta(1^{\ses,+}_{(1,n)}) = 
  1^{\ses,+}_{(1,n)} \otimes 1 
  + \sum_{l \ge 0} \wt{\theta}^+_l k c^{n-l/2} \otimes 1^{\ses,+}_{(1,n-d)}
$$
with
$$
 \wt{\theta}^\pm_d := \theta^\pm_d c^{\mp d/2}.
$$
The formula \eqref{eq:g=0:hp:t-th} is rewritten as 
$$
 (\wt{t}_d,\wt{\theta}_l)_H = \delta_{d,l} \dfrac{[2d]_{\sqrt{\bbL}}}{d}
$$
Now by Definition \ref{dfn:d_red-U} of the reduced Drinfeld double,
the algebra $\tDU{\bbP^1}$ is described in the following way.

\begin{prop}[{\cite[\S4]{BS:2013}}]\label{prop:g=0:DU-rel}
$\tDU{\bbP^1}$ is the associative algebra with the generators 
$$
 1^{\ses,\pm}_{(1,n)} \ (n \in \bbZ),\quad
 \wt{t}^{\pm}_d \ (d \in \bbZ_{\ge 1}), \quad
 k^{\pm1}, \quad
 c^{\pm 1/2}
$$
modulo  the relation 
\begin{enumerate}
\item
$c^{\pm 1/2}$ are central, $c^{1/2}c^{-1/2}=1$ and  
$k k^{-1} = 1 = k^{-1} k$.

\item 
$[k,\wt{t}^\pm_l] = 0 = [\wt{t}^\pm_l,\wt{t}^\pm_m]=0$
for any $l,m \in \bbZ_{\ge 1}$.

\item
$k 1^{\ses,\pm}_{(1,n)} = \bbL^{\mp1} 1^{\ses,\pm}_{(1,n)} k$ 
for any $n \in \bbZ$. 

\item
For any $d \in \bbZ_{\ge 1}$ and $n \in \bbZ$ we have
$$
 [\wt{t}^\pm_d, 1^{\ses,\pm}_{(1,n)}] 
 = \dfrac{[2 d]}{d} 1^{\ses,\pm}_{(1,n+d)} c^{\mp d/2}.
$$

\item
For any $d \in \bbZ_{\ge 1}$ and $n \in \bbZ$ we have
$$
 [\wt{t}^\pm_d, 1^{\ses,\mp}_{(1,n)}] 
 = -\dfrac{[2 d]}{d} 1^{\ses,\mp}_{(1,n-d)} c^{\pm d/2}.
$$

\item
For any $d,l \in \bbZ_{\ge 1}$ we have
$$
 [\wt{t}^+_d, \wt{t}^\mp_l] = 
 \delta_{d,l} \dfrac{[2 d]}{d} \dfrac{c^{-d}-c^{d}}{\bbL^{1/2}-\bbL^{-1/2}}.
$$

\item
For $n,m \in \bbZ$ we have
\begin{align*}
[1^{\ses,+}_{(1,m)}, 1^{\ses,-}_{(1,n)}] = 
 \begin{cases}
  \dfrac{1}{1-\bbL} \wt{\theta}^+_{n-m} k c^{(m+n)/2} & \text{if $n>m$} \\
  0 & \text{if $n=m$} \\
  \dfrac{1}{\bbL-1} \wt{\theta}^-_{m-n} k^{-1} c^{-(m+n)/2} & \text{if $n<m$} 
 \end{cases}.
\end{align*}
\end{enumerate}
\end{prop}

Although we have already described the quantum loop algebra $U_v(\Lsl)$ 
in Fact \ref{fct:Dr:UqLg:double} as the Drinfeld double of $U_v(\Lsl)$,
have we dare to write down the definition for the clarification.
Recall Definition \ref{dfn:UvLb} where we set $R \subset \bbQ(v)$ 
consisting of the rational functions having poles only at $0$ or at roots of $1$.

\begin{dfn}\label{dfn:g=0:UqLg}
$U_v(\Lsl)$ is the associative algebra over $R$ generated by 
$$
 E^\pm_n \ (n \in \bbZ), \quad
 H_r \ (r \in \bbZ \setminus \{0\}), \quad 
 K^{\pm1}, \quad C^{\pm 1/2}
$$
subject to the following relations.
\begin{enumerate}
\item 
$C^{\pm 1/2}$ are central, $C^{1/2}C^{-1/2}=1$ and 
$K K^{-1} = K^{-1}K = 1$.

\item
$[K,H_r]=0$ for any $r \in \bbZ\setminus \{0\}$.

\item
$[K,E^\pm_n]=v^{\mp2}E^\pm_n K$ for any $n \in \bbZ$.
 
\item
For any $r,s \in \bbZ\setminus\{0\}$ we have
$$
 [H_r,H_s] = \delta_{r+s,0}\dfrac{[2r]_v}{r}\dfrac{C^m-C^{-m}}{v-v^{-1}}.
$$

\item
For any $n \in \bbZ$ and $r \in \bbZ\setminus\{0\}$ we have
$$
 [H_r,E^\pm_n] = \pm \dfrac{[2r]_v}{r}E^\pm_{n+r}C^{\mp |r|/2}.
$$

\item
For any $m,n \in \bbZ$ we have
$$
 E^\pm_m E^\pm_{n+1} + E^\pm_n E^\pm_{m+1}
=v^{\pm2}(E^\pm_{n+1} E^\pm_m + E^\pm_{m+1} E^\pm_n).
$$

\item
For any $m,n \in \bbZ$ we have
$$
 [E^+_m,E^-_n] = 
 \dfrac{v}{v-v^{-1}}\bigl(\Psi^+_{m+n}C^{(m-n)/2} - \Psi^-_{m+n}C^{(n-m)/2}  \bigr)
 K^{\sgn(m+n)},
$$
where $\sgn(n)$ denotes the sign of the integer $n$,
and $\Psi^{\pm}_{\pm d}$ ($d \in \bbZ_{\ge 1}$) are given by
$$
 1+ \sum_{d \ge 1} \Psi^\pm_{\pm d} z^d = 
  \exp\Bigl(\pm(v^{-1}-v)\sum_{d \ge 1} H_{\pm d} z^d \Bigr)
$$
and $\Psi^{\pm}_{\mp d}=0$ ($d\in\bbZ_{\ge 1}$).
\end{enumerate}
\end{dfn}

\begin{rmk}
We followed \cite[Definition 4.9]{BS:2013}.
The choice of generators are crucial for the construction of the isomorphism 
explained below.
\end{rmk}

Now we can state the main result of this section.
Recall that $\wt{\bbQ} := \bbQ[v^{\pm 1}]/(v^{-2} - \bbL)$.

\begin{thm}[{\cite[Proposition 4.11]{BS:2013}}]
Let $U_{\sqrt{\bbL}}(\Lsl) := U_v(\Lsl) \otimes_R \wt{\bbQ}$.
Then the map $U_{\sqrt{\bbL}}(\Lsl) \to \tDU{\bbP^1}$ given by 
$$
 E^\pm_n   \longmapsto 1^{\ses,\pm}_{(1, \pm n)}        \ (n \in \bbZ),\quad
 H_{\pm d} \longmapsto \pm \wt{t}^\pm_{\pm d} \ (d \in \bbZ_{\ge 1}),\quad
 K       \longmapsto k,\quad
 C^{1/2} \longmapsto c^{1/2} 
$$
is an isomorphism of algebras.
\end{thm}

\begin{proof}
One can show that the map sends $\Psi^\pm_{\pm l}$ to $\wt{\theta}\pm_l$ 
for any $l \in \bbZ_{\ge 1}$.
Then by comparing the relations in Proposition \ref{prop:g=0:DU-rel} 
and Definition \ref{dfn:g=0:UqLg}, one finds that the map is well-defined.

Now the map is checked to be an isomorphism with the help of the basis consisting of 
the elements of the form 
$E^+_\lambda H_{\mu^+} K^a C^{b/2} E^-_\nu H_{\mu^-}$,
where $\lambda$, $\mu^\pm$, $\nu$ are non-decreasing finite sequences of integers, 
$E^+_\lambda := E^+_{\lambda_1} E^+_{\lambda_2} \cdots$ and so on,
and $a,b \in \bbZ$.
We omit the detail and refer the proof of \cite[Proposition 4.11]{BS:2013}.
\end{proof}

%%%%%%%%%%%%%%%%%%%%%%%%%%%%%%%%%%%%%%%%%%%%%%%%%%
%%%%%%%%%%%%%%%%%%%%%%%%%%%%%%%%%%%%%%%%%%%%%%%%%%
%%%%%%%%%%%%%%%%%%%%%%%%%%%%%%%%%%%%%%%%%%%%%%%%%%

\section{The case of an irreducible projective curve of arbitrary arithmetic genus}
\label{sect:Carb}

In this section we denote by $\Ca$  an irreducible reduced projective curve $\Ca$ 
defined over a field $\fk$,
and consider the motivic Hall algebra of $\Ca$.
In the smooth case our result is just a motivic restatement
of the one given in \cite{K:1997} and explained in \cite[\S4.11]{S:lect}.

%%%%%%%%%%%%%%%%%%%%%%%%%%%%%%%%%%%%%%%%%%%%%%%%%%
%%%%%%%%%%%%%%%%%%%%%%%%%%%%%%%%%%%%%%%%%%%%%%%%%%
\subsection{The composition subalgebra}

As mentioned at Lemma \ref{lem:He:Gamma-graded}, 
the Hall algebra $\Hall{\Ca}$ and 
the extended Hall algebra $\He{\Ca}$ are $\Gamma$-graded,
where $\Gamma \subset \Num(\Ca)$ is the submonoid 
generated by effective classes.
We have 
$$
 \Num(\Ca) \simto \bbZ^2,\quad
 \ol{\shE} \longmapsto \bigl(\rk(\shE),\deg(\shE)\bigr)
$$
as modules,
and we denote by $k_{(r,d)}$ with $(r,d) \in \bbZ^2$ 
the element of the extension part of $\He{\Ca}$.
The Riemann-Roch theorem yields
$$
 \mathop{\chi}(\shE,\shF)
=(1-g)\rk(\shE) \rk(\shF)  + \rk(\shE) \deg(\shF) - \rk(\shF) \deg(\shE),
$$
where $g$ is the arithmetic genus of $\Ca$.

As in the case of $\bbP^1$,
the algebra $\Hall{\Ca}$ or $\He{\Ca}$ is too big to treat,
and we will mainly study the composition subalgebra $\U{\Ca}$
and its extended version $\Ue{\Ca}$ given in Definition \ref{dfn:comp-sa}.

Let us denote by $\stkM_{(0,d)}$ the moduli stack of perfect torsion sheaves 
of degree $d$,
and by $\stkM_{\lf,(1,n)}$ the moduli stack of line bundles of degree $n$.
These are substacks of $\stkM = \stkM(\Ca)$, 
the moduli stack of perfect coherent sheaves on $\Ca$.

By definition, 
$\U{\Ca}$ is the subalgebra of $\Hall{\Csm}$ generated by elements
$$
 1_{(0,d)} := [\stkM_{(0,d)} \hookrightarrow \stkM] \quad (d \in \bbZ_{\ge 1})
$$
and
$$
 1^{\ses}_{(1,n)} := [\stkM_{\lf,(1,n)} \hookrightarrow \stkM] \quad (n \in \bbZ).
$$

We will replace the generators $1_{(0,d)}$ by the following elements $t_d$.

\begin{dfn}
Define $t_d$ ($d \in \bbZ_{\ge1}$) and $\theta_l$ ($l \in \bbZ_{\ge0}$) 
by the generating functions
$$
 1 + \sum_{d \ge 1} 1_{(0,d)} z^d = 
 \exp\Bigl(\sum_{d \ge 1}\dfrac{t_d}{[d]_{\sqrt{\bbL}}} z^d\Bigr)
$$
and
$$
 \sum_{l \ge 0} \theta_l z^l = 
 \exp\Bigl((\bbL^{1/2}-\bbL^{-1/2})
  \sum_{d \ge 1} \dfrac{t_d}{[d]_{\sqrt{\bbL}}} z^d\Bigr).
$$
\end{dfn}

As stated in \cite[Lemma 4.51]{S:lect}, one can calculate coproducts of these generators.

\begin{lem}
\label{lem:Carb:copro}
For $d \in \bbZ_{\ge1}$ and $n \in \bbZ$ we have
\begin{align*}
&\Delta(t_d) = t_d \otimes 1 + k_{(0,d)} \otimes t_d,\\
&\Delta(1^{\ses}_{(1,n)}) = 1^{\ses}_{(1,n)} \otimes 1 
   + \sum_{l \ge 0}\theta_l k_{(1,n-l)} \otimes 1^{\ses}_{(1,n-l)}.
\end{align*}
\end{lem}

\begin{proof}
This is the same as Fact \ref{fct:g=0:Delta}.
As mentioned in Remark \ref{rmk:g=0:Delta},
the proof works for an arbitrary curve.
\end{proof}

As for the Hopf pairing $(\cdot,\cdot)_H$, we have the following formulas.

\begin{lem}
\label{lem:Carb:Hpairing}
For any $d,l \in \bbZ_{\ge1}$ and $m,n \in \bbZ$ we have
\begin{align*}
&(t_d,1^{\ses}_{(1,n)})_H = 0,\\
&(t_d,t_l)_H = \delta_{d,l}\dfrac{[d]_{\sqrt{\bbL}}}{d} \dfrac{[{\Ca}_{(d)}]}{\bbL^d-1},\\
&(1^{\ses}_{(1,m)},1^{\ses}_{(1,n)})_H
  = \delta_{m,n}\dfrac{[\Pic^n(\Ca)]}{\bbL - 1}.
\end{align*}
Here $[\Ca_{(d)}] \in K(\Var{\fk})$ is defined by 
$$
 \zeta_{\mot}(\Ca;z) = \exp\Bigl(\sum_{d\ge1}[\Ca_{(d)}] \dfrac{z^d}{d}\Bigr)
$$
\end{lem}

\begin{proof}
The smooth case is stated in \cite{K:1997} 
and explained in {\cite[Lemma 4.52]{S:lect}}.
We copy the proof here.

By Definition \ref{dfn:H-pairing} of the pairing $(\cdot,\cdot)_H$,
the first is obvious.

As mentioned in Remark \ref{rmk:g=0:HopfPairing},
the proof of Fact \ref{fct:g=0:HopfPairing} works in this case,
and we have the second formula.

For the third formula, note that 
$1^{\ses}_{(1,n)} = \sum_{\shL} [\stkM_{\shL} \hookrightarrow \stkM]$,
where $\shL$ runs over the isomorphism classes of line bundles of degree $n$,
and $\stkM_{\shL}$ denotes the stack for the line bundles isomorphic to $\shL$.
Since $\Aut(\shL) \simeq \fk^\times$ so that 
$a_{\shL} = [\Aut(\shL)] = \bbL - 1 \in K(\Var{\fk})$, we have the result.
\end{proof}

We close this subsection 
by explaining Kapranov's result of expressing 
the relation of generators of the composition subalgebra 
in the current form.

\begin{prop}\label{prop:Ca:current}
Set 
$$
 x^+(z) := \sum_{n \in \bbZ} 1^{\ss}_{(1,n)} z^n,\quad
 \psi(z) := \exp\Bigl(\sum_{d\ge1}\dfrac{t_{(0,d)}}{[d]_{\sqrt{\bbL}}}z^d\Bigr).
$$
Then in the algebra $\U{\Ca}$ the following relations hold.
\begin{align*}
\zeta_{\mot}(C;w/z)x^+(z)x^+(w) &= \zeta_{\mot}(C;w/z)x^+(z)x^+(w),
\\
\psi(z) x^+(w) &= x^+(w)\psi(z),
\\
\psi(z)\psi(w) &= \psi(w)\psi(z).
\end{align*}
\end{prop}

%%%%%%%%%%%%%%%%%%%%%%%%%%%%%%%%%%%%%%%%%%%%%%%%%%%%%%%%%%%%%%%%%%%%%%
%%%%%%%%%%%%%%%%%%%%%%%%%%%%%%%%%%%%%%%%%%%%%%%%%%%%%%%%%%%%%%%%%%%%%%

\subsection{The slope stability of coherent sheaves on curves}
\label{subsec:stability}

For the future use, in particular for the discussion of the Hall algebras 
for elliptic curves,
let us recall the notion of slope stability of coherent sheaves 
on a projective curve.
The slope stability was originally introduced by Mumford  
in the smooth case to construct coarse moduli spaces of 
semistable coherent sheaves on curves 
using geometric invariant theory.
Here we give statements for arbitrary (not necessarily smooth) projective curves,
following Simpson's generalization \cite{Si:1994}.
Among a large amount of literature on this topic, 
we only cite \cite{HL:book} for the detail.

In this subsection $\Ca$ denotes a projective curve 
defined over a field $\fk$.
$\Coh(\Ca)$ denotes the category of coherent sheaves on $\Ca$.
For $\shE \in \Coh(\Ca)$, 
the rank $\rk(\shE)$ and the degree $\deg(\shE)$ is defined 
by $\rk(\shE) := \alpha_1(\shE)/\alpha_1(\shO_{\Ca})$ and 
$\deg(\shE) := \alpha_0(\shE)/\alpha_1(\shO_{\Ca})$,
where $\alpha_i(-) \in \bbZ$ is given by the equation
$\chi(- \otimes \shO_{\Ca}(m H)) = m \alpha_1(-) + \alpha_0(-)$.
Here $H$ denotes the fixed ample line bundle of $\Ca$. 
By the Riemann-Roch theorem, 
this definition coincides with the classical one 
in the case of smooth curves.

\begin{dfn}
\begin{enumerate}
\item 
For $E \in \Coh({\Ca})$, the \emph{slope} $\mu(\shE)$ is the rational number 
defined by
$$
 \mu(\shE) := \dfrac{\deg(\shE)}{\rk(\shE)} \in \bbQ \cup \{\infty\},
$$
where if $\rk(\shE)=0$ then we set $\mu(\shE) := \infty$.

\item
A coherent sheaf $\shE$ is called \emph{semistable} if 
$\mu(\shF) \le \mu(\shE)$ holds for any non-zero subsheaf $\shF \subset \shE$.
$\shE$ is called \emph{stable} if $\mu(\shF)<\mu(\shE)$  holds for 
any non-zero proper subsheaf $\shF \subsetneq \shE$.

\item
For $\nu \in \bbQ \cup \{\infty\}$, denote by $\catS_\nu$ 
the full subcategory of $\Coh({\Ca})$ consisting of semistable sheaves  
of slope $\nu$.
\end{enumerate}
\end{dfn}

In order to distinguish this notion of stability from others,
we call it \emph{slope stability}.
We obviously have $\catS_{\infty} = \Tor({\Ca})$, 
the category of torsion sheaves on ${\Ca}$. 
Here are some fundamental properties of semistable sheaves.
 
\begin{fct}\label{fct:stability}
\begin{enumerate}
\item 
$\Hom(\catS_{\nu},\catS_{\nu'})=0$ if $\nu>\nu'$.

\item
For any $\nu \in \bbQ \cup \{\infty\}$, 
the category $\catS_\nu$ is abelian and closed under extension.
Moreover it is a finite length category,
i.e., every object has a finite composition sequence with simple factors,
and the simple objects are stable sheaves.
\end{enumerate}
\end{fct}

Let us also recall the Harder-Narasimhan filtration.

\begin{fct}
For any $\shE \in \Coh({\Ca})$ there exists a unique filtration
$$
 0 = \shE_0 \subset \shE_1 \subset \cdots \shE_n = \shE
$$
such that 
every factor $\shF_i := \shE_i/\shE_{i-1}$ is semistable and 
$\mu(\shF_1) > \mu(\shF_2) > \cdots > \mu(\shF_n)$.
It will be called the \emph{Harder-Narasimhan filtration} of $\shE$.
\end{fct}

By the existence and uniqueness of the Harder-Narasimhan filtration, 
we immediately have the following statement.

\begin{cor}\label{cor:HN}
The motivic Hall algebra for ${\Ca}$ has a decomposition 
$$
 \Hall{\Ca} \longsimto 
  \bigoplus_n \bigoplus_{\mu_1>\mu_2>\cdots>\mu_n} 
  \Hall{\Ca}_{\mu_1}\otimes \Hall{\Ca}_{\mu_2} \otimes \cdots \otimes \Hall{\Ca}_{\mu_n}
$$
as $\bbK$-module.
Here for $\nu \in \bbQ \cup \{\infty\}$ 
$\Hall{\Ca}_\nu$ denotes $K(\St{\stkM_\nu})$,
the subalgebra generated by  perfect semistable coherent sheaves with slope $\nu$.
($\stkM_\nu$ denotes the moduli stack of perfect semistable coherent sheaves with slope $\nu$.)
\end{cor}

%%%%%%%%%%%%%%%%%%%%%%%%%%%%%%%%%%%%%%%%%%%%%%%%%%
%%%%%%%%%%%%%%%%%%%%%%%%%%%%%%%%%%%%%%%%%%%%%%%%%%

\section{The case of a smooth elliptic curve}
\label{sect:g=1:sm}

We now recall the work of Burban and Schiffmann \cite{BS:2012} 
on the Hall algebra of a smooth elliptic curve.
Our review will be done in the motivic version, but the argument 
is essentially the same with theirs.
Hereafter $\Esm$ denotes a smooth elliptic curve defined over 
a fixed field $\fk$.

%%%%%%%%%%%%%%%%%%%%%%%%%%%%%%%%%%%%%%%%%%%%%%%%%%
%%%%%%%%%%%%%%%%%%%%%%%%%%%%%%%%%%%%%%%%%%%%%%%%%%

\subsection{Generalities on coherent sheaves on a smooth elliptic curve}

The Euler pairing on $K(\Coh(\Esm))$ is given by 
$$
 \mathop{\chi}(\ol{\shE},\ol{\shF}) = 
  \rk(\ol{\shE})\deg(\ol{\shF})-\rk(\ol{\shF})\deg(\ol{\shE}).
$$
As for the numerical Grothendieck group $\Num(\Esm)$ 
(see \eqref{eq:NumX} for the definition),
we have
$$
 \Num(\Esm) \longsimto \bbZ^2
$$ 
by the map 
$$
 \ol{\shE} \longmapsto \bigl(\rk(\shE),\deg(\shE)\bigr).
$$

Let us recall the Mumford stability explained 
in the previous subsection \S\ref{subsec:stability}.
The existence of Harder-Narasimhan filtration and Corollary \ref{cor:HN} 
tell us that the subcategory $\catS_\nu$ of semistable sheaves of slope $\nu$
is the building block of the motivic Hall algebra $\Hall{\Esm}$.
Now it is time to recall the fundamental result of Atiyah \cite{A} 
on the classification of semisimple sheaves on elliptic curves.
We cite the result in the following form \cite[Theorem 4.45]{S:lect}.

\begin{fct}[{\cite{A}}]\label{fct:Atiyah}
For any $\nu,\nu' \in\bbQ \cup \{\infty\}$,
there is an equivalence of categories 
$\Phi_{\nu,\nu'}: \catS_\mu \simto \catS_{\nu'}$.
In particular, $\catS_\mu \simto \catS_{\infty} = \Tor(\Esm)$.
\end{fct}

\begin{rmk}
The equivalence can be realized as a Fourier-Mukai transform \cite{Mukai:1981} 
on $\DCoh(\Esm)$,
which will be explained in \S\ref{subsec:auto-FM}.
\end{rmk}

Let us recall that $\Hall{\Esm}_\mu = K(\St{\stkM_\nu})$ denotes  
the subalgebra of $\Hall{\Esm}$ generated by the semistable sheaves of slope $\nu$
(see Corollary \ref{cor:HN}).
Fact \ref{fct:Atiyah} implies that 
the equivalence $\phi_{\nu,\infty}$ induces an isomorphism between 
$\Hall{\Esm}_\infty$ and $\Hall{\Esm}_\nu$.

\begin{dfn}\label{dfn:g=1:UX}
\begin{enumerate}
\item 
For $\nu \in \bbQ$ denote by 
$\phi_{\infty,\nu}: \Hall{\Esm}_\infty \simto \Hall{\Esm}_\nu$ 
the isomorphism of bialgebras induced by the equivalence $\Phi_{\infty,\nu}$.

\item
For $\nu \in \bbQ$ and $d \in \bbZ_{\ge 1}$, 
define $t_{\nu,d} := \phi_{\infty,\nu}(t_d)$,
where $t_d \in \Hall{\Esm}_\infty = \Hall{\Esm}_\tor$ is given in Definition \ref{dfn:td}.

\item
Denote by $\U{\Esm}$ the subalgebra of $\Hall{\Esm}$ generated by 
$\{t_{\nu,d} \mid \nu \in \bbQ \cup \{\infty\}, \ d \in \bbZ_{\ge 1}\}$.
Also we define $\U{\Esm}_\nu := \U{\Esm} \cap \Hall{\Esm}_\nu$ for 
any $\nu \in \bbQ \cup \{\infty\}$.
\end{enumerate}
\end{dfn}

The last definition looks natural from the view of Atiyah's result (Fact \ref{fct:Atiyah}).
Previously (Definition \ref{dfn:comp-sa}) 
we defined the composition algebra $\U{\Esm}$ as the subalgebra of $\Hall{\Esm}$ 
generated by torsion sheaves and line bundles.
Here we have 

\begin{fct}[{\cite{SV:2011}}]
$\U{\Esm}$ in Definition \ref{dfn:g=1:UX} coincides
with the one in Definition \ref{dfn:comp-sa} for a smooth elliptic curve $\Esm$.
\end{fct}

\begin{cor}[{\cite{BS:2012}}]
\begin{enumerate}
\item 
The multiplication map induces an isomorphism 
$$
 \bigoplus_n \bigoplus_{\mu_1>\mu_2>\cdots>\mu_n} 
 \U{\Esm}_{\mu_1} \otimes \U{\Esm}_{\mu_2} \otimes \cdots \otimes \U{\Esm}_{\mu_n} 
  \longsimto \U{\Esm}
$$
of $\bbK$-modules.

\item
$\U{\Esm}_\nu \simeq \fk[t_{\nu,1},t_{\nu_2},\ldots]$ for any $\nu \in \bbQ \cup \{\infty\}$.
\end{enumerate}
\end{cor}

\begin{proof}
(1)
As in Corollary \ref{cor:HN}, the existence of Harder-Narasimhan filtration 
implies the result.

(2)
By the isomorphism 
$$
 \phi_{\infty,\nu}|_{\U{\Esm}_\infty}: \U{\Esm}_\infty \simto \U{\Esm}
$$ 
and the description 
$$
 \U{\Esm}_\infty = \fk[t_1,t_2,\ldots],
$$
we have the result.
\end{proof}

%%%%%%%%%%%%%%%%%%%%%%%%%%%%%%%%%%%%%%%%%%%%%%%%%%
%%%%%%%%%%%%%%%%%%%%%%%%%%%%%%%%%%%%%%%%%%%%%%%%%%

\subsection{The composition subalgebra for smooth elliptic curve}

\cite{BS:2012} gave a presentation of the double of the algebra $\U{\Esm}$ 
which is $\bbZ^2$-graded and $\SL(2,\bbZ)$-invariant.
Let us state it in the motivic language.
Hereafter we assume $\fk$ is algebraically closed,
so that the formula in Fact \ref{fct:Mustata:thm} of 
Kapranov's motivic zeta function $\zeta_{\mot}(\Esm;z)$ applies.
For an elliptic curve $X$ it is expressed as  
$$
 \zeta_{\mot}(\Esm;z) = 1 + \dfrac{[\Pic^0(\Esm)] z}{(1-z)(1- \bbL z)}
                   = \dfrac{(1-q_1 z)(1-q_2 z)}{(1-z)(1- \bbL z)},
$$
where $q_1$ and $q_2$ are conjugate in an algebraic extension of $K(\Var{\fk})$,
and satisfy $q_1 q_2 = \bbL$ and $q_1+q_2=\bbL+1-[\Pic^0(\Esm)]$.

Following \cite{BS:2012}, we introduce several symbols.
We will often use the subset 
$$
 (\bbZ^2)^+ := 
 \bigl(\bbZ_{\ge 1} \times \bbZ\bigr) \sqcup \bigl(\{0\} \times \bbZ_{\ge1}\bigr)
$$
of $\bbZ^2$ for the grading of $\U{\Esm}$.

\begin{dfn}\label{dfn:g=1:BS-symbols1}
\begin{enumerate}
\item
For $i \in \bbZ$ set 
$$
 c_i := (q_1^{i/2} - q_1^{-i/2})(q_2^{i/2} - q_2^{-i/2}) \dfrac{[i]_{\sqrt{\bbL}}}{i}.
$$ 

\item
Recalling $t_{\nu,d}$ in Definition \ref{dfn:g=1:UX},
set 
$t_{(r,d)} := t_{d/r,\gcd(r,d)}$ 
for $(r,d) \in (\bbZ^2)^+$.

\item
For $\bx = (r,d) \in (\bbZ^2)^+$, set 
\begin{align*}
 \gcd(\bx) := 
 \begin{cases} 
  \gcd(r,d) & \text{if $r \neq 0$} \\
  d         & \text{if $r = 0$}
 \end{cases}.
\end{align*}

\item
For $\bx = (r,d) \in (\bbZ^2)^+$, set 
\begin{align*}
 \mu(\bx) := 
  \begin{cases} 
   d/r &\text{if $r\neq 0$} \\ 
   \infty & \text{if $r = 0$}  
  \end{cases}.
\end{align*}

\item
For $\bx,\by \in (\bbZ^2)^+$, set 
$\ep(\bx,\by) := \sgn(\det(\bx,\by)) \in \{\pm 1\}$.
\end{enumerate}
\end{dfn}

\begin{rmk}
$\gcd(\bx)$ was denoted by $\deg(\bx)$ in \cite{BS:2012}.
\end{rmk}

Obviously $\U{\Esm}$ is generated by the elements 
$\{t_{(r,d)} \mid (r,d) \in (\bbZ^2)^+\}$.

\begin{fct}[{\cite{BS:2012}}]
\label{fct:Esm:upper}
By the assignment $t_{r,d} \mapsto T_{(r,d)}$ for $(r,d) \in (\bbZ^2)^+$, 
the composition subalgebra $\U{\Esm}$ is isomorphic to 
the associative algebra generated by the elements
$$
 \{ T_{r,d} \mid (r,d) \in (\bbZ^2)^+ \} 
$$
subject to the following relations.
\begin{enumerate}
\item 
If $\bx,\by \in (\bbZ^2)^+$ satisfy $\mu(\bx) = \mu(\by)$,
then $[T_{\bx}, T_{\by}]=0$,

\item
If $\bx,\by \in (\bbZ^2)^+$ satisfy $\gcd(\bx)=1$ and 
if there is no interior lattice point in the triangle 
with vertices $\mathbf{0}$, $\bx$ and $\bx+\by$ in $\bbZ^2$,
then
$$
 [t_{\by},t_{\bx}] = \ep_{\bx,\by} c_{\gcd(\by)} 
 \dfrac{\theta_{\bx+\by}}{\bbL^{1/2}-\bbL^{-1/2}},
$$
where $\theta_{\bx}$'s are defined by the generating series
$$
 \sum_{i \ge 1} {\theta_{i \bx_0}} z^i
 = \exp\Bigl((\bbL^{1/2}-\bbL^{-1/2})\sum_{j \ge 1}t_{j \bx_0} z^j \Bigr)
$$ 
for $\bx_0 \in (\bbZ^2)^+$ with $\gcd(\bx_0) =1$.
\end{enumerate}
\end{fct}

We will not give the detail here.
The proof in \cite{BS:2012} utilizes $\SL(2,\bbZ)$-action 
on the reduced Drinfeld double of $\U{\Esm}$,
which is explained in the next subsection.

%%%%%%%%%%%%%%%%%%%%%%%%%%%%%%%%%%%%%%%%%%%%%%%%%%
%%%%%%%%%%%%%%%%%%%%%%%%%%%%%%%%%%%%%%%%%%%%%%%%%%
\subsection{The Drinfeld double}

The reduced Drinfeld double $\Dr{\U{\Esm}}$ 
of the $\U{\Esm}$ for a smooth elliptic curve $\Esm$ 
is determined by \cite{BS:2012}.
It has the grading with respect to the subset 
$$
 (\bbZ^2)^* := \bbZ^2 \setminus \{(0,0)\} = (\bbZ^2)^+ \sqcup (\bbZ^2)^-,
$$
of $\bbZ^2$, where 
$(\bbZ^2)^- := -(\bbZ^2)^+ = 
 \{(r,d) \in \bbZ^2 \mid r < 0 \text{ or } (r = 0 \text{ and } d < 0) \}$.
We will use the notation $\bx=(r,d)$ for an element of $(\bbZ^2)^*$,
and define $\deg(\bx) \in \bbZ$ 
and $\mu(\bx) \in \bbQ \sqcup \{\pm\infty\}$ as in Definition \ref{dfn:g=1:BS-symbols1}.
Precisely speaking, we have 

\begin{dfn}
\begin{enumerate}
\item 
For $\bx = (r,d) \in (\bbZ^2)^*$, set 
\begin{align*}
 \gcd(\bx) := 
 \begin{cases} 
  \gcd(r,d) & \text{if $r \neq 0$} \\
  d         & \text{if $r = 0$}
 \end{cases}.
\end{align*}
Thus we have $\deg(\bx) \in \bbZ_{> 0}$ for $\bx \in (\bbZ^2)^+$ 
and $\deg(\bx) \in \bbZ_{< 0}$ for $\bx \in (\bbZ^2)^-$.

\item
For $\bx = (r,d) \in (\bbZ^2)^*$, set 
\begin{align*}
 \mu(\bx) := 
 \begin{cases}
   d/r  &\text{if $r\neq 0$} \\ 
   \infty  & \text{if $r = 0$ and $d > 0$}  \\
   -\infty & \text{if $r = 0$ and $d < 0$}  
 \end{cases}.
\end{align*}

\item
For $\bx,\by \in (\bbZ^2)^*$, set 
$\ep(\bx,\by) := \sgn(\det(\bx,\by)) \in \{\pm 1\}$.

\item
\begin{align*}
 \ep(\bx) := 
 \begin{cases}
    1  & \text{if $x \in (\bbZ^2)^+$} \\ 
   -1  & \text{if $x \in (\bbZ^2)^-$}    
 \end{cases}.
\end{align*}
\end{enumerate}
\end{dfn}

\begin{fct}[{\cite{BS:2012}}]
\label{fct:Esm:double}
$\Dr{\U{\Esm}}$ is isomorphic to the associative algebra  
generated by the elements $\{t_{(r,d)} \mid (r,d) \in (\bbZ^2)^*\}$
and $\{k_{r,d} \mid r,d \in \bbZ\}$ 
modulo the following relations.
\begin{enumerate}
\item 
$k_{r,d} k_{r',d} = k_{r+r',d+d'}$,

\item
If $\bx,\by \in (\bbZ^2)^*$ satisfy $\mu(\bx)=\mu(\by)$, then
$$
 [t_{\bx}, t_{\by}] = \delta_{\bx,-\by} c_{\gcd(\bx)} 
  \dfrac{k_{\bx} - k_{\bx}^{-1}}{\bbL^{1/2}-\bbL^{-1/2}}.
$$

\item
If $\bx,\by \in (\bbZ^2)^*$ satisfy $\gcd(\bx)=1$ and 
if there is no interior lattice point in the triangle 
with vertices $\mathbf{0}$, $\bx$ and $\bx+\by$ in $\bbZ^2$,
then
$$
 [t_{\by}, t_{\bx}] = \ep_{\bx,\by} c_{\gcd(\by)} 
 k_{\alpha(\bx,\by)} \dfrac{\theta_{\bx+\by}}{\bbL^{1/2}-\bbL^{-1/2}},
$$
where $\alpha(\bx,\by)$ is given by
\begin{align*}
 \alpha(x,y) := 
 \begin{cases}
  \dfrac{1}{2}\ep_{\bx}
  \bigl(\ep_{\bx} \bx + \ep_{\by} \by - \ep_{\bx+\by} (\bx+\by)\bigr)
  & \text{if $\ep_{\bx,\by}=1$} \\
  \dfrac{1}{2}\ep_{\by}
  \bigl(\ep_{\bx} \bx + \ep_{\by} \by - \ep_{\bx+\by} (\bx+\by)\bigr)
  & \text{if $\ep_{\bx,\by}=-1$} 
 \end{cases}
\end{align*}
and $\theta_{\bx}$ is determined by 
$$
 \sum_{i \ge 1} {\theta_{i \bx_0}} z^i
 = \exp\Bigl((\bbL^{1/2}-\bbL^{-1/2})\sum_{j \ge 1}t_{j \bx_0} z^j \Bigr)
$$ 
for $\bx_0 \in (\bbZ^2)^+$ with $\gcd(\bx_0) =1$.
\end{enumerate}
\end{fct}

As mentioned in the previous subsection.
the proof in \cite{BS:2012} 
utilizes the $\SL(2,\bbZ)$-action.
Notice that the relations of the algebra in the above Fact 
are obviously $\SL(2,\bbZ)$-(quasi)invariant.
The $\SL(2,\bbZ)$-action on $\Dr{\U{\Esm}}$ 
comes from the autoequivalences of $\DCoh(\Esm)$,
which will be explained in the next subsection.

Finally we relate $\Dr{\U{\Esm}}$ to the quantum toroidal algebra for $\fgl_1$.

\begin{fct}[{\cite[Theorem 5.4]{BS:2012}}]\label{fct:BS:E-U}
The algebra given in Fact \ref{fct:Esm:double} 
is isomorphic to the $\fgl_1$-quantum toroidal algebra $\algU$ 
after tensoring coefficient field.
Thus we have 
$$
 \Dr{\U{\Esm}} \longsimto \algU \otimes_{\bbQ(q_1,q_2)}\bbK.
$$
\end{fct}

Here recall that $\bbK = K(\St{\fk})[\bbL^{\pm1/2}]$,
and $\algU$ is considered as an algebra defined over 
$\bbQ(q_1,q_2)$ with $q_3 = q_1^{-1}q_2^{-1}$.
The ring homomorphism $\bbQ(q_1,q_2) \to \bbK$ 
is given by mapping $q_1$ and $q_2$ to the inverse of zeros of 
the motivic zeta function
$$
 \zeta_\mot(\Esm;z) 
 = \dfrac{1 + \bigl(\bbL+1-[\Pic^0(\Esm)]\bigr) z + \bbL z^2}{(1-z)(1-\bbL z)}
 = \dfrac{(1-q_1 z)(1-q_2 z)}{(1-z)(1-\bbL z)}.
$$

\begin{proof}
Let us explain the outline of the proof explained in \cite{BS:2012}. 
Notice that the relations in Proposition \ref{prop:Ca:current} 
coincides with those of $\algU$ except the Serre-like relation.
Let us denote by $\algU'$ the algebra generated by the same generators 
in $\algU$ satisfying the relations except the Serre-like relation.
Denote by $\wt{\mathrm{U}}$ the algebra given in Fact \ref{fct:Esm:double}.
Then one can find that there is a surjection 
$\algU' \twoheadrightarrow \wt{\mathrm{U}}$.
By the analysis of the convex paths in $\bbZ^2$ 
(corresponding to Harder-Narasimhan filtrations of coherent sheaves on $\Esm$),
one can show that the Serre-like relation coincides with the kernel of 
the above surjection.
See \cite{S:2012} for the precise account for this analysis.
\end{proof}

%%%%%%%%%%%%%%%%%%%%%%%%%%%%%%%%%%%%%%%%%%%%%%%%%%
%%%%%%%%%%%%%%%%%%%%%%%%%%%%%%%%%%%%%%%%%%%%%%%%%%
\subsection{The automorphism}
\label{subsec:auto-FM}

The Drinfeld double $\Dr{\U{\Esm}}$,
or more generally $\Dr{\U{\Csm}}$ for a smooth curve $\Csm$, 
may be considered as the Hall algebra of the root category 
$\catR(\Esm) := \DCoh(\Esm)/[1]^2$,
where $[1]$ denotes the shift of complexes,
and the quotient symbol means taking the orbit category.
This idea is realized in the recent work of Bridgeland \cite{B:2013}.
In this paper we will not treat this approach
and use the definition of Drinfeld double itself.

Let us recall that the description of 
the group $\Aut(\DCoh(\Esm))$ of auto-equivalences of $\DCoh(\Esm)$
for a smooth elliptic curve $\Esm$ defined over a field $\fk$.
We will also mention some generalities on Fourier-Mukai transforms.
The basic references are \cite{BBH:book} and \cite{H:book}.

By the fundamental result of Orlov \cite{O:1997},
For any smooth projective variety $X$ 
defined over a algebraically closed field $\fk$,
every auto-equivalence of $\DCoh(X)$  
can be realized as a \emph{Fourier-Mukai transform} \cite{Mukai:1981}.
Let $Y$ be another smooth projective variety and 
$\shF$ be an object of $\DCoh(X \times Y)$,
namely a bounded complex of coherent sheaves on the product variety $X \times Y$.
Then the functor $\Phi_{\shF}: \DCoh(X) \to \DCoh(Y)$ is defined to be 
$$
 \Phi_{\shF}(-) := 
  \mathop{\mathbf{R} p_{Y *}}
   \bigl(\shF \stackrel{\mathbf{L}}{\otimes} p_X^* (-)\bigr),
$$
where $p_X: X \times Y \to X$ and $p_Y: X \times Y \to Y$ 
are natural projections 
and the symbols $\mathop{\mathbf{R} p_{Y *}}$ 
and $\stackrel{\mathbf{L}}{\otimes}$ denote the derived functors.
We call this functor a Fourier-Mukai transform if 
it gives an equivalence of categories.

It is also known by \cite{O:2002} that for a (smooth) abelian variety $A$, 
any smooth projective variety $X$ with equivalent derived category
$\DCoh(X) \simeq \DCoh(A)$ 
is also a smooth abelian variety of the same dimension, 
and moreover if $\Esm$ is a smooth elliptic curve,
then $X$ is isomorphic to $\Esm$.
Thus $\Aut\bigl(\DCoh(\Esm)\bigr)$ is generated by Fourier-Mukai transforms 
$\Phi_{\shF}$ with $\shF \in \DCoh(\Esm \times \Esm)$.

For any smooth projective variety $Y$,
we have natural equivalences of $\DCoh(Y)$ given by 
\begin{itemize}
\item
the shift $[1]$ of complexes,
\item
automorphism of $Y$,
\item
tensoring with a line bundle of degree $0$.
\end{itemize}
These can be also realized by Fourier-Mukai transforms,
and consist the subgroup $\bbZ \oplus \Aut(Y) \ltimes \Pic^0(Y)$ 
of $\Aut(\DCoh(Y))$.
Let us define the group $\FM(Y)$ by the short exact sequence 
\begin{align}\label{eq:g=1:FM}
\xymatrix{
  1 \ar[r] & \bbZ \oplus \Aut(Y) \ltimes \Pic^0(Y) \ar[r] 
  & \Aut\bigl(\DCoh(Y)\bigr) \ar[r] & \FM(Y) \ar[r] & 1}.
\end{align}

If $A$ is an abelian variety,
then a set of generators of $\FM(A)$ is given by $\{\Phi_{\shF}\}$,
where $\shF$'s are the universal family of semi-homogeneous sheaves.

In the case of a smooth elliptic curve $\Esm$,
the description of $\FM(\Esm)$ is extremely simple,
and the result is 
\begin{align}\label{eq:g=1:FM-SL}
  \FM(\Esm) = \langle T_{\shO}, \Phi_{\shP} \rangle \simto \SL(2,\bbZ).
\end{align}
Here we denoted by $T_{\shO}$ 
the \emph{spherical twist} by the structure sheaf $\shO=\shO_{\Esm}$,
which is defined by
$$
 T_{\shO}(-) := \Cone\bigl(\RHom(\shO,-) \otimes \shO \longto{\text{ev}} - \bigr).
$$
We also denoted by $\Phi_{\shP_{\Esm}}$ the Fourier-Mukai transforms 
defined by the Poincare bundle $\shP_{\Esm}$ on $\Esm \times \Pic^0(\Esm)$.
($\shP_X$ is the universal family of the Jacobian variety $\Pic^0(X)$).

One can understand the isomorphism \eqref{eq:g=1:FM-SL} 
in terms of K-theoretic Fourier-Mukai transforms $\Phi^K_{\shF}$,
which is the automorphism on the Grothendieck group 
$$
 K(\DCoh(\Esm)) = K(\Coh(\Esm))
$$ 
induced by $\Phi_{\shF}$.
We can further consider the induced automorphisms 
on the numerical Grothendieck group $\Num(\Esm)$.

Denote it by $t_{\shO}$ and $\phi_{\shP}$ the automorphisms on $\Num(\Esm)$ 
induced by $T_{\shO}$ and $\Phi_{\shP}$.
Since $\Num(\Esm) \simeq \bbZ^2$,
they can be realized as matrices. 
The result is 
$$
 t_{\shO}    = \begin{pmatrix} 1 & -1 \\ 0 & 1 \end{pmatrix},\quad
 \phi_{\shP} = \begin{pmatrix} 0 & -1 \\ 1 & 0 \end{pmatrix}.
$$
In fact. this argument gives us a homomorphism $\FM(\Esm) \to \SL(2,\bbZ)$ 
of groups, 
and one can show that it is indeed an isomorphism.

Now note that the part $\Aut(\Esm) \ltimes \Pic^0(\Esm)$ 
acts trivially on $\Dr{\U{\Esm}}$.
Thus it is natural to consider the action of 
$\bbZ$ (generated by shifts of complexes) 
and $\FM(\Esm)$ on on $\Dr{\U{\Esm}}$.
By the description \eqref{eq:g=1:FM} and \eqref{eq:g=1:FM-SL},
these parts coincide with 
the universal over $\wh{\SL}(2,\bbZ)$ of $\SL(2,\bbZ)$, 
which sits in the short exact sequence 
\begin{align*}
\xymatrix{
  1 \ar[r] & \bbZ \ar[r] & \wh{\SL}(2,\bbZ) \ar[r] 
& \SL(2,\bbZ) \ar[r] & 1},   
\end{align*}
acts on the algebra $\Dr{\U{\Esm}}$.

To write down the $\wh{\SL}(2,\bbZ)$-action, 
we need to introduce the \emph{winding number}.
Note that we can lift the natural $\SL(2,\bbZ)$-action 
on the circle $S^1 = \bbP \bbR^1 = (\bbR^2 \setminus \{0\}) / \bbR_{>0}$ 
to an $\wh{\SL}(2,\bbZ)$-action on $\bbR$ 
by the identification $S^1 = \bbR / 2\bbZ$.

\begin{dfn}
\begin{enumerate}
\item
For a $(r,d) \in \bbZ^*$, denote by $(r:d)$ the corresponding element 
of $S^1 =\bbP \bbR^1$,
and by $\ol{(r:d)} \in \bbR$ any lift of $(r:d)$.

\item
For a slope $d/r \in \bbQ \sqcup \{\pm \infty\}$ 
and an element $\gamma \in \wh{\SL}(2,\bbZ)$,
define the winding number $w(\gamma,d/r) \in \bbZ$ by
\begin{align*}
w(\gamma,d/r):=
 \begin{cases}
 \#\Bigl( \bbZ \cap \bigl[\ol{(r:d)},\gamma(\ol{(r:d)})\bigr] \Bigr)
 & \text{if $\ol{(r:d)} \le \gamma(\ol{(r:d)})$} \\
 - \#\Bigl( \bbZ \cap \bigl[\ol{(r:d)},\gamma(\ol{(r:d)})\bigr] \Bigr)
 & \text{otherwise}
 \end{cases}.
\end{align*}
Here $\bigl[\ol{(r:d)},\gamma(\ol{(r:d)})\bigr]$ denotes the interval 
between $\ol{(r:d)}$ and $\gamma(\ol{(r:d)})$ in $\bbR$.
\end{enumerate}
\end{dfn}

\begin{fct}[{\cite[(6.16)]{BS:2012}}]
The group $\wh{\SL}(2,\bbZ)$ acts on $\Dr{\U{\Esm}}$ by 
$$
 \gamma(k_{\bx}) = k_{\gamma(\bx)},\quad 
 \gamma(t_{\bx}) = t_{\Phi(x)} k_{\gamma(\bx)}^{w(\gamma,\mu(x))}
$$
for $\gamma \in \SL(2,\bbZ)$,
\end{fct}

\begin{rmk}
For the ordinary Hall algebra $\cal{H}_{\ext}(\catA)$ 
of a finitary abelian category $\catA$,
the existence of the action of $\Aut(\mathop{\mathsf{D^b}}\catA)$ 
on the Drinfeld double 
of the algebra is shown by Cramer \cite{C:2010}. 
\end{rmk}

Finally we can state the second main result of \cite{BS:2012}.

\begin{fct}[{\cite{BS:2012}}]
The Fourier-Mukai transform 
$\Phi_{\shP}$ with the Poinvare bundle as its kernel
induces the algebra automorphism 
$\Phi^H = \theta$ on $\Dr{\U{\Esm}} = \algU \otimes \bbK$.
\end{fct}

%%%%%%%%%%%%%%%%%%%%%%%%%%%%%%%%%%%%%%%%%%%%%%%%%%
%%%%%%%%%%%%%%%%%%%%%%%%%%%%%%%%%%%%%%%%%%%%%%%%%%
%%%%%%%%%%%%%%%%%%%%%%%%%%%%%%%%%%%%%%%%%%%%%%%%%%

\section{The case for irreducible projective curves of arithmetic genus one}
\label{sect:Ea}

In this section we will give the main result of this paper.
We investigate the motivic-Hall algebras 
for singular irreducible projective curves of arithmetic genus one,
namely, a singular elliptic curve with a node or a cusp.
The final result will be quite similar to the smooth elliptic case.

In this section $\fk$ denotes a fixed  algebraically closed field of characteristic zero.
$C$ denotes an irreducible reduced projective curve over $\fk$,
so that it may be singular.

%%%%%%%%%%%%%%%%%%%%%%%%%%%%%%%%%%%%%%%%%%%%%%%%%%
%%%%%%%%%%%%%%%%%%%%%%%%%%%%%%%%%%%%%%%%%%%%%%%%%%
\subsection{Semistable sheaves on singular elliptic curves and Fourier-Mukai transforms}
\label{subsec:Ea:FM}

We follow \cite{BK:2006} for the description of 
semistable sheaves on $C$.

%As in the case of smooth curves, 
The numerical Grothendieck group $\Num(C)$ is isomorphic to 
$\bbZ^2$ as modules by the homomorphism
$$
 \Num(C) \longsimto \bbZ^2,\quad
 \overline{\shE} \longmapsto \bigl(\rk(\shE),\deg(\shE)\bigr).
$$
Recalling \S\ref{subsec:stability},
we denote by $\mu(\shE) := \deg(\shE)/\rk(\shE)$ 
the slope of $\shE \in \Coh(C)$,
and consider the slope stability of coherent sheaves. 
Denote by $\catS_\nu$ the full subcategory of $\Coh(C)$ 
consisting of semistable sheaves of slope $\nu$.

Here is the fundamental fact on the semistable sheaves on $C$.
%Let us denote by 
%$$
% Q := 
% \{\nu \in \bbQ \cup \{\infty\} 
%   \mid \text{$\catS_\nu$ contains a non-zero object} \}.
%$$

\begin{fct}[{\cite[Corolary 4.3]{BK:2006}}]
\label{fct:BK}
For any $\nu \in \bbQ \cup \{\infty\}$, 
the category $\catS_\nu$ is equivalent to the category 
$\catS_{\infty} = \Tor(C)$ of torsion sheaves. 
\end{fct}

The equivalence is realized by Foruier-Mukai transform.
Recall the exact sequence \eqref{eq:g=1:FM}.
\begin{align*}
\xymatrix{
  1 \ar[r] & \bbZ \oplus \Aut(C) \ltimes \Pic^0(C) \ar[r] 
  & \Aut\bigl(\DCoh(C)\bigr) \ar[r] & \FM(C) \ar[r] & 1}.
\end{align*}
By \cite{BK:2005},
the subgroup of $\FM(C)$ generated by 
the spherical twists $T_{\shO_C}$ 
and $T_{\shO_{x_0}}$,
where $x_0$ is a regular point of $C$,
is isomorphic to $\SL(2,\bbZ)$.

In the category $\catS_\infty = \Tor(C)$,
the stable sheaf is precisely the structure sheaves $\shO_x$ of 
closed points $x$ of $C$.
If $C$ is singular with one singular point $s$,
then all the structure sheaf $\shO_s$ is the unique non-perfect stable sheaf 
in $\catS_\infty$.

As a consequence of Fact \ref{fct:BK},
for each $\nu \in \bbQ \cup \{\infty\}$ 
there is precisely one object of $\catS_\nu$ 
which is stable but not perfect.

%%%%%%%%%%%%%%%%%%%%%%%%%%%%%%%%%%%%%%%%%%%%%%%%%%%
%%%%%%%%%%%%%%%%%%%%%%%%%%%%%%%%%%%%%%%%%%%%%%%%%%%
\subsection{Motivic zeta function}

Here we study the motivic zeta functions 
for irreducible curve of arithmetic genus one.
Recall that $\fk$ denotes an algebraically closed field of characteristic zero,
and $\Ea$ denotes an irreducible reduced projective curve of arithmetic genus $1$ 
over $\fk$.

\begin{prop}
\label{prop:E:zeta}
The motivic zeta function $\zeta_{\mot}(\Ea;z)$ of the curve $\Ea$ 
has the form
$$
 \zeta_{\mot}(\Ea;z) = \dfrac{1 + a z + \bbL z^2}{(1-z)(1-\bbL z)}
$$
with $a \in K(\Var{\fk})$.
\end{prop}

\begin{proof}
It follows from the result of \cite{AP:1996},
where the result was shown for the ordinary Hasse-Weil zeta function
in the case when $E$ is defined on the finite field $\bbF_q$.
\end{proof}

%%%%%%%%%%%%%%%%%%%%%%%%%%%%%%%%%%%%%%%%%%%%%%%%%%%
%%%%%%%%%%%%%%%%%%%%%%%%%%%%%%%%%%%%%%%%%%%%%%%%%%%
%
%\subsection{Relative Fourier-Mukai transforms}
%
%We review the formalism of relative Fourier-Mukai transforms 
%following \cite{B:1998} and \cite[Chapter 6]{BBH:book}.
%It will be used for the investigation of 
%motivic Hall algebras for the singular fibers of 
%relatively minimal elliptic surfaces.

%%%%%%%%%%%%%%%%%%%%%%%%%%%%%%%%%%%%%%%%%%%%%%%%%%
%%%%%%%%%%%%%%%%%%%%%%%%%%%%%%%%%%%%%%%%%%%%%%%%%%
\subsection{The composition subalgebra}

Now we consider the motivic Hall algebra for a projective curve $\Ea$ of arithmetic genus $1$.
Recall Definition \ref{dfn:comp-sa} in \S\ref{subsec:comp-subalg}
where we defined the composition subalgebra $\U{\Ea}$.
On the reduced Drinfeld double of $\U{\Ea}$,
we have a similar result as in the case of a smooth elliptic curve.
Let $\algU$ be the $\fgl_1$-quantum toroidal algebra 
considered to be defined over $\bbQ(q_1,q_2)$.

\begin{thm}\label{thm:Ea:qt}
$\Dr{\U{\Ea}}$ is isomorphic to $\algU \otimes_{\bbQ(q_1,q_2)} \bbK$.
Here the ring homomorphism $\bbQ(q_1,q_2) \to \bbK = K(\St{\fk})$ 
is defined by mapping $q_1$ and $q_2$ to the inverse of the zeros 
of the motivic zeta function of $\Ea$.
Namely we put 
$$
 \zeta_{\mot}(\Ea;z) = \dfrac{(1-q_1 z)(1-q_2 z)}{(1-z)(1-\bbL z)}.
$$
\end{thm}

\begin{proof}
As mentioned at Fact \ref{fct:Esm:double},
the isomorphism $\Dr{\U{\Esm}} \simeq \algU \otimes \bbK$,
comes from the $\SL(2,\bbZ)$-action.
As we observed in \S\ref{subsec:Ea:FM}.
we have that action on $\Dr{\U{\Ea}}$.
Proposition \ref{prop:Ca:current} claims that 
the relation on the composition subalgebra 
only depends on the motivic zeta function $\zeta-a;z)$.
Since by Proposition \ref{prop:E:zeta} $\zeta(\Ea;z)$ and $\zeta(\Esm;z)$ 
have the same form,
we know that the generators of $\U{\Ea}$ satisfy the same relations 
as $\U{\Esm}$.
Thus $\Dr{\U{\Ea}}$ has the same description as $\Dr{\U{\Esm}}$ has in 
Fact \ref{fct:Esm:double}.
\end{proof}

By the comparison of the Fourier-Mukai transforms,
we also have

\begin{thm}\label{thm:Ea:theta}
There is a Fourier-Mukai transform 
inducing the algebra automorphism 
$\theta_1$ on the algebra $\Dr{\U{\Ea}} = \algU \otimes \bbK$.
\end{thm}

%%%%%%%%%%%%%%%%%%%%%%%%%%%%%%%%%%%%%%%%%%%%%%%%%%
%%%%%%%%%%%%%%%%%%%%%%%%%%%%%%%%%%%%%%%%%%%%%%%%%%
%%%%%%%%%%%%%%%%%%%%%%%%%%%%%%%%%%%%%%%%%%%%%%%%%%


\begin{thebibliography}{NNNNN2}

\bibitem[A57]{A}
M.~F.~Atiyah, 
\emph{Vector bundles over an elliptic curve}, 
Proc.\ Lond.\ Math.\ Soc.\ (3) \textbf{7} (1957), 
414--452.

\bibitem[AP96]{AP:1996}
Y.~Aubry, M.~Perret,
\emph{A Weil theorem for singular curves},
in \emph{Arithmetic, geometry and coding theory} (Luminy, 1993), 1--7, 
de Gruyter (1996). 

\bibitem[BBH]{BBH:book}
C.~Bartocci, U.~Bruzzo, D.~Hern\'{a}ndez Ruip\'{e}rez,
\emph{Fourier-Mukai and Nahm transforms in geometry and mathematical physics},
Progress in Mathematics \textbf{276}, Birkh\"{a}user (2009).

\bibitem[BaK01]{BK:2001}
P.~Baumann,  C.~Kassel, 
\emph{The Hall algebra of the category of coherent sheaves on the projective line}, 
J.\ Reine Angew.\ Math.\ \textbf{533} (2001), 207--233.


%\bibitem[Br98]{B:1998}
%T.~Bridgeland, 
%\emph{Fourier-Mukai transforms for elliptic surfaces}, 
%J.\ reine angew.\ math.\ \textbf{498} (1998),  115--133.

\bibitem[Br12]{B:2012}
T.~Bridgeland,
\emph{An introduction to motivic Hall algebras},
Adv.\ Math.\ \textbf{229} (2012), 102--138.

\bibitem[Br13]{B:2013}
T.~Bridgeland, 
\emph{Quantum groups via Hall algebras of complexes},
Ann.\ of Math.\ (2) \textbf{177} (2013), no. 2, 739--759. 

\bibitem[BuK05]{BK:2005}
I.~Burban, B.~Kreu\ss ler, 
\emph{Fourier-Mukai transforms and semi-stable sheaves on nodal
Weierstra\ss  cubics}, 
J.\ reine angew.\ Math.\ \textbf{584} (2005), 45--82.

\bibitem[BuK06]{BK:2006}
I.~Burban, B.~Kreu\ss ler
\emph{Derived categories of irreducible projective curves of arithmetic genus one},
Compos.\ Math.\ \textbf{142} (2006), no.\ 5, 1231--1262. 

\bibitem[BuS12]{BS:2012}
I.~Burban, O.~Schiffmann,
\emph{On the Hall algebra of an elliptic curve I},
Duke Math.\ J.\ \textbf{161} (2012), no.\ 7, 1171--1231.

\bibitem[BuS13]{BS:2013}
I.~Burban, O.~Schiffmann,
\emph{Two descriptions of the quantum affine algebra 
$U_v(\widehat{\mathfrak{sl}}_2)$ via Hall algebra approach},
Glasgow Math.\ J.\ \textbf{54} (2012), 283--307.


\bibitem[C10]{C:2010}
T.~Cramer, 
\emph{Double Hall algebras and derived equivalences},
Adv.\ Math.\ \textbf{224} (2010), no.\ 3, 1097--1120. 

%\bibitem[DG01]{DG:2001}
%Yu.~Drozd, G.-M.~Greuel, 
%\emph{Tame and wild projective curves and classification of vector bundles}, 
%J.\ Algebra \textbf{246} (1) (2001), 1--54.

\bibitem[D86]{D:1986}
V.~Drinfeld,
\emph{Quantum Groups}, 
in ICM proceedings, Berkeley (1986), 798--820.


\bibitem[FHHSY]{FHHSY}
B.~Feigin, K.~Hashizume, A.~Hoshino, J.~Shiraishi, S.~Yanagida,
\emph{A commutative algebra on degenerate $\bbC\bbP^1$ and Macdonald polynomials},
J.\ Math.\ Phys.\ \textbf{50} (2009), no.\ 9, 095215, 42 pp.

\bibitem[FFJMM1]{FFJMM1}
B.~Feigin, E.~Feigin, M.~Jimbo, T.~Miwa, E.~Mukhin, 
\emph{Quantum continuous $\frk{gl}_{\infty}$: Semi-infinite
construction of representations}, 
Kyoto J.\ Math.\ \textbf{51} (2011), no.\ 2, 337--364.

\bibitem[FFJMM2]{FFJMM2}
B.~Feigin, E.~Feigin, M.~Jimbo, T.~Miwa, E.~Mukhin, 
\emph{Quantum continuous $\frk{gl}_{\infty}$: 
Tensor product of Fock modules and $W_n$ characters}, 
Kyoto J.\ Math.\ \textbf{51} (2011), no.\ 2, 365--392.

\bibitem[FJMM]{FJMM:2013}
B.~Feigin, M.~Jimbo, T.~Miwa, E.~Mukhin
\emph{Branching rules for quantum toroidal $\frk{gl}_n$},
preprint (2013), arXiv:1309.2147.


\bibitem[FT11]{FT:2011}
B.~Feigin, A.~Tsymbaliuk,
\emph{Equivariant $K$-theory of Hilbert schemes via shuffle algebra}, 
Kyoto J.\ Math.\ \textbf{51} (2011), no.\ 4, 831--854.

\bibitem[GKV95]{GKV:1995}
V.~Ginzburg, M.~Kapranov, E.~Vasserot, 
\emph{Langlands reciprocity for algebraic surfaces}, 
Math.\ Res.\ Lett.\ \textbf{2} (1995), no.\ 2, 147--160.

\bibitem[G95]{G:1995}
J.~Green,
\emph{Hall algebras, hereditary algebras and quantum groups},
Invent.\ Math.\ \textbf{120} (1995), no.\ 2, 361--377.

%\bibitem[HLST09]{HLST:2009}
%D.~Hern\'{a}ndez Ruip\'{e}rez, A.~C.~L\'{o}pez Mart\'{i}n, D.~S\'{a}nchez G\'{o}mez,
%C.~Tejero Prieto,
%\emph{Moduli Spaces of Semistable Sheaves on Singular Genus 1 Curves},
%Internat.\ Math.\ Res.\ Notices, \textbf{2009} (2009), no.\ 23, 4428--4462

\bibitem[H]{H:book}
D.~Huybrechts,
\emph{Fourier-Mukai transforms in algebraic geometry}, 
Oxford Math.\ Mon., Oxford University Press (2006).

\bibitem[HL]{HL:book}
D.~Huybrechts, M.~Lehn,
\emph{The geometry of moduli spaces of sheaves}, 
2nd ed., Cambridge University Press (2010). 

\bibitem[Jos95]{Jos:1995}
A.~Joseph, 
\emph{Quantum groups and their primitive ideals},
Springer (1995).


\bibitem[Joy06a]{Joy:2006:JLMS}
D.~Joyce,
\emph{Constructible functions on Artin stacks}, 
J.\ London Math.\ Soc.\ (2) \textbf{74} (2006), no.\ 3, 583--606.


\bibitem[Joy07a]{Joy:2007:QJM}
D.~Joyce,
\emph{Motivic invariants of Artin stacks and stack functions}, 
Q.\ J.\ Math.\ \textbf{58} (2007), 345--392.

\bibitem[Joy06b]{Joy:2006:I}
D.~Joyce,
\emph{Configurations in abelian categories. I.\ Basic properties and moduli stacks}, 
Adv.\ Math.\ \textbf{203} (2) (2006), 194--255.

\bibitem[Joy07b]{Joy:2007:II}
D.~Joyce,
\emph{Configurations in abelian categories. II.\ Ringel-Hall algebras}, 
Adv.\ Math.\ \textbf{210} (2) (2007), 635--706.

\bibitem[JS12]{JS:2012}
D.~Joyce, Y.~Song,
\emph{A theory of generalized Donaldson-Thomas invariants}, 
Mem.\ Amer.\ Math.\ Soc.\ \textbf{217} (2012), no.\ 1020.

\bibitem[K97]{K:1997}
M.~Kapranov,
\emph{Eisenstein series and quantum affine algebras}, 
J.\ Math.\ Sci.\ (New York) \textbf{84} (1997),  no.\ 5, 1311--1360.

\bibitem[K]{K}
M.~Kapranov,
\emph{The elliptic curve in the S-duality conjecture and Eisenstein series for Kac-Moody groups},
preprint (2000),
arXiv:math.AG/0001005v2.

\bibitem[Mac95]{M:1995}
I.G.~Macdonald, 
\emph{Symmetric functions and Hall polynomials}, 2nd edition, 
Oxford Math.\ Mon., Oxford University Press (1995).

\bibitem[Mi07]{Mi:2007}
K.~Miki, 
\emph{A $(q,\gamma)$-analog of the $W_{1+\infty}$ algebra}, 
J.\ Math.\ Phys.\ \textbf{48} (2007), no.\ 12, 1--35.

%\bibitem[Mo10]{Moz:2010}
%S.~Mozgovoy
%\emph{Classification of semistable sheaves on a rational curve with one node},
%J.\ Algebra \textbf{323} (2010), no.\ 1, 14--26. 

\bibitem[Muk81]{Mukai:1981}
S.~Mukai,
\emph{Duality between $D(X)$ and $D(\widehat{X})$ with 
its application to Picard sheaves},
Nagoya Math.\ J.\ \textbf{81} (1981), 153--175. 

\bibitem[Mus]{Mu}
M.~Mustata,
\emph{Zeta functions in algebraic geometry}, Lecture Note (2011),
available at the webpage 
http://www.math.lsa.umich.edu/\~{}mmustata/.

\bibitem[O97]{O:1997}
D.~O.~Orlov,
\emph{Equivalences of derived categories and K3 surfaces}, 
J.\ Math.\ Sci.\ (New York) \textbf{84} (1997), no.\ 5, 1361--1381.

\bibitem[O02]{O:2002}
D.~O.~Orlov,
\emph{Derived categories of coherent sheaves on abelian varieties and 
equivalences between them}
Izv.\ Math.\ \textbf{66} (2002), no.\ 3, 569--594 

\bibitem[R90]{R:1990}
C.~Ringel,
\emph{Hall algebras and quantum groups}, 
Invent.\ Math.\ \textbf{101} (1990), no.\ 3, 583--591.

\bibitem[S]{S:lect}
O.~Schiffmann,
\emph{Lectures on Hall algebras},
in \emph{Geometric methods in representation theory. II}, 1--141, 
S\'emin.\ Congr., 24-II, Soc.\ Math.\ France, Paris (2012);
arXiv:math/06116172v2. 

\bibitem[S12]{S:2012}
O.~Schiffmann,
\emph{Drinfeld realization of the elliptic Hall algebra}, 
J.\ Alg.\ Combin.\ \textbf{35} (2012), 237--262.

\bibitem[SV11]{SV:2011}
O.~Schiffmann, E.~Vasserot, 
\emph{The elliptic Hall algebra, Cherednick Hecke algebras and 
Macdonald polynomials},
Compos.\ Math.\ \textbf{147} (2011), 188--234.

\bibitem[ST01]{ST:2001}
P.~Seidel, R.~Thomas,
\emph{Braid group actions on derived categories of coherent sheaves}, 
Duke Math.\ J.\ \textbf{108} (2001), 37--108.

\bibitem[Si94]{Si:1994}
C.~T.~Simpson,
\emph{Moduli of representations of the fundamental group of 
a smooth projective variety. 1},  
Publ.\ Mat., I.H.E.S.\  \textbf{79} (1994), 47--129.


\bibitem[T]{T}
B.~To\"{e}n,
\emph{Grothendieck rings of Artin $n$-stacks},
preprint (2005),
arXiv:math/050908.

\bibitem[X97]{X:1997}
J.~Xiao, 
\emph{Drinfeld double and Ringel-Green theory of Hall algebras}, 
J.~Algebra \textbf{190}  (1997), no.\ 1, 100--144.
\end{thebibliography}
\end{document}